\documentclass[reqno]{amsart}
\usepackage{amssymb,latexsym,amsmath,amsthm,enumerate}
\usepackage[mathscr]{eucal}
\usepackage{framed,color,graphicx}
\usepackage{mathrsfs}
\usepackage[all]{xy}
\usepackage{tikz}
\usetikzlibrary{positioning}
\tikzset{%
element/.style={draw, shape=circle, fill=white, inner sep=1.4pt}
}

\DeclareSymbolFont{bbold}{U}{bbold}{m}{n}
\DeclareSymbolFontAlphabet{\mathbbold}{bbold}

\theoremstyle{plain}
\newtheorem{thm}{Theorem}[section]
\newtheorem{lem}[thm]{Lemma}
\newtheorem{cor}[thm]{Corollary}
\newtheorem{pro}[thm]{Proposition}
\newtheorem{example}[thm]{Example}

\newtheorem{problem}[thm]{Problem}

\theoremstyle{definition}
\newtheorem{defn}[thm]{Definition}

\newtheorem{remark}[thm]{Remark}

\newcommand{\occ}{\operatorname{occ}}

\newcommand{\up}[1]{\textup{#1}}

\newcommand{\onto}{\twoheadrightarrow}

\newcommand{\ba}{\mathbf{a}}

\newcommand{\bu}{\mathbf{u}}
\newcommand{\bv}{\mathbf{v}}
\newcommand{\bw}{\mathbf{w}}

\begin{document}

\title[Ai-semirings]{Nonfinitely based ai-semirings with finitely based semigroup reducts}

\author{Marcel Jackson}
\address{Department of Mathematics and Statistics\\ La Trobe University\\ Victoria  3086\\
Australia} \email{M.G.Jackson@latrobe.edu.au}
\author{Miaomiao Ren} \address{School of Mathematics, Northwest University, Xi'an, 710127, Shaanxi, P.R. China}\email{miaomiaoren@yeah.net}
\author{Xianzhong Zhao}
\address{School of Mathematics, Northwest University, Xi'an, 710127, Shaanxi, P.R. China}\email{zhaoxz@nwu.edu.cn}

\keywords{Semiring, finite basis problem, \texttt{NP}-complete, variety membership}

\thanks{Miaomiao Ren is supported by National Natural Science Foundation of China (11701449). Xianzhong Zhao is supported by National Natural Science Foundation of China (11971383, 11571278).}
\begin{abstract}
We present some general results implying nonfinite axiomatisability of many additively idempotent semirings with finitely based semigroup reducts.  The smallest is a $3$-element commutative example, which we show also has \texttt{NP}-hard membership for its variety.  As well as being the only nonfinite axiomatisable ai-semiring on $3$-elements, we are able to show that its nonfinite basis property infects many related semirings, including the natural ai-semiring structure on the semigroup $B_2^1$.  We also extend previous group-theory based examples significantly, by showing that any finite additively idempotent semiring with a nonabelian nilpotent subgroup is not finitely axiomatisable for its identities.
\end{abstract}

\maketitle
\section{Introduction}
A \emph{semiring} is an algebra with two associative binary operations $+,\cdot$, in which $+$ is commutative and $\cdot$ distributes over $+$ from the left and right.  Sometimes an additive and multiplicative $0$ element is required, though in the theme of many earlier investigations in the theory of semiring varieties, we do not ask this here.  The special case of \emph{additively idempotent semirings} (henceforth, \emph{ai-semirings}), where $+$ is idempotent, has received particular  attention in the study of varieties and equational properties; these are also often called \emph{semilattice ordered semigroups}.  Several of the most famous semirings are ai-semirings: the Kleene semiring of regular languages (see Conway \cite{con} for example), the max-plus and min-plus semirings of tropical analysis (see Aceto, \'Esik and Ing\'olfsd\'ottir~\cite{AEI}, for example),  the powerset semirings of semigroups (see Dolinka~\cite{dol1} for example), and semirings of binary relations under composition and union (Andr\'eka and Mikul\'as \cite{andmik} for example).  The $+$ operation is interpreted as union in all these cases except for the max-plus and min-plus algebras where it is max and min, respectively.   There is a further combinatorial appeal to ai-semirings, as distributivity of $\cdot$ over $+$ and the commutative and idempotent properties of $+$ easily reveal that ai-semiring terms correspond in natural ways to sets of (multiplicative) semigroup terms.  This already suggests a strong affinity between semigroup varieties and semiring varieties, and so it is perhaps not surprising that the vast majority of investigations relating to the general theory of ai-semiring varieties have built on the very heavily developed theory of semigroup varieties.  We have in mind in particular, results concerning the nonfinite axiomatisability properties of ai-semirings.  

An algebra $A$, or variety $V$ is said to be \emph{finitely based} (FB) if the set of its valid equations is derivable from some finite subset; otherwise it is said to be \emph{nonfinitely based} (NFB).  A series of papers by Igor Dolinka  \cite{dol0,dol1,dol2,dol3} showed that early ideas of Peter Perkins~\cite{per} for finite semigroups and the very deep contributions of Mark Sapir \cite{sap1,sap2} for inherently nonfinitely based semigroups (precise definition later) could be translated, with some care, to some significant and natural examples of ai-semirings.  An exception to this theme are the known results of flat extensions of finite groups, which are FB always as semigroups (they are simply Clifford semigroups), yet are FB as semirings only when the underlying group omits any nonabelian nilpotent subgroups, Jackson \cite[Theorem 7.3]{jac:flat}.  In this instance it is instead the quasi-equational theory of the underlying multiplicative semigroup that translates to the equational theory of the semiring.  All of the other example references \cite{AEI,andmik,con} in the opening paragraph of the article also consider NFB  issues for these classically arising ai-semirings, though the focus in the present article will be on finite semirings.

In the meantime, a quite extensive investigation into syntactically defined classes of ai-semirings has been performed, predominantly relating to establishing finite axiomatisability and describing variety lattices.  Early work by McKenzie and Romanowska \cite{mckrom} showed that when multiplication (as well as addition) is idempotent and commutative, there is always a finite equational basis.  The broader class of semilattice-ordered bands is considered in \cite{GPZ,pas,paszha}, where it is ultimately shown that there are precisely 78 varieties of multiplicatively idempotent ai-semirings, all having the FB property.  Additively idempotent semirings satisfying $x^n\approx x$ were studied by Ku\v{r}il and Pol\'ak \cite{kurpol}, and there have been significant advances, particularly in the commutative case. The second and third authors \cite{renzha} showed that there are 9 distinct varieties of ai-semirings satisfying $x^3\approx x, xy\approx yx$, and more recently with Shao \cite{RZS} have  extended this to higher periods, showing that when $n-1$ is square free, the lattice of subvarieties of the variety of ai-semirings defined by $x^n\approx x, xy\approx yx$ has $2+2^{r+1}+3^r$ elements, where $r$ denotes the number of prime divisors of $n-1$.  Dropping the commutative condition, the second and third authors with Wang \cite{RZW} identify exactly 179 ai-semiring varieties satisfying $x^3\approx x$.   Very recently, the second and third authors with Crvenkovi\'c, Shao and Dapi\'c \cite{ZRCSD} have provided a detailed investigation of ai-semirings of order $3$, revealing that of the 61 possibilities (as well as 6 of order $2$), all but perhaps one has a finite basis for its identities.  The remaining unresolved semiring is denoted by $S_7$ in the classification there and has the following table, and is the starting point for the contributions of the present article.
\begin{center}
\begin{tabular}{c|ccc}
$+$&1&$a$&0\\
\hline
$1$&1&$0$&0\\
$a$&0&$a$&0\\
$0$&0&$0$&0\\
\end{tabular}\qquad
\begin{tabular}{c|ccc}
$\cdot$&1&$a$&0\\
\hline
$1$&1&$a$&0\\
$a$&$a$&$0$&0\\
$0$&0&$0$&0\\
\end{tabular}
\end{center}
The multiplicative semigroup reduct of $S_7$ is a seemingly innocuous commutative monoid, with identity basis $x^2\approx x^3,xy\approx yx$, so that it is not initially an obvious candidate for challenging equational behaviour as a semiring.  Despite this we find that not only does  $S_7$ have the NFB property, it holds this property rather infectiously.  We derive a quite general condition that implies the NFB property provided that $S_7$ is contained in the variety.  This is held by many examples that, like $S_7$, are FB as multiplicative semigroups, as well as others.  One consequence of particular interest is the natural semiring structure on the combinatorial Brandt monoid $B_2^1$, which we show is NFB.  Multiplicatively, $B_2^1$ can be given by the following matrices under matrix multiplication:
\[
\begin{tabular}{cccccc}
$\left(\begin{matrix} 0&0\\0&0\end{matrix}\right)$
&
$\left(\begin{matrix} 1&0\\0&1\end{matrix}\right)$
&
$\left(\begin{matrix} 0&1\\0&0\end{matrix}\right)$
&
$\left(\begin{matrix} 0&0\\1&0\end{matrix}\right)$
&
$\left(\begin{matrix} 1&0\\0&0\end{matrix}\right)$
&
$\left(\begin{matrix} 0&0\\0&1\end{matrix}\right)$
\\
\rule{0cm}{.5cm}$0$&$1$&$a$&$b$&$ab$&$ba$
\end{tabular}
\]
There is a unique semilattice order on $B_2^1$ that makes it a semiring, and it is precisely the order that arises from the usual natural order as an inverse semigroup, albeit using $\geq$ rather than $\leq$; see Figure \ref{fig:B21}.
\begin{figure}
\begin{tikzpicture}
\node [element,fill=black] (0) at (2.5,2) [label=$0$]  {};
\node [element,fill=black] (ab) at (2,1) [label=below left:$ab$]  {};
\node [element,fill=black] (ba) at (3,1) [label=below right:$ba$]  {};
\node [element,fill=black] (1) at (2.5,0) [label=right:$1$]  {};
\node [element] (a) at (1,1) [label=left:$a$]  {};
\node [element] (b) at (4,1) [label=right:$b$]  {};
\draw (0) -- (ab) -- (1) -- (ba)--(0)--(a);
\draw (0) -- (b);
\end{tikzpicture}
\caption{The semigroup $B_2^1$ becomes a semiring if $+$ is the  join operation given by this order.}\label{fig:B21}
\end{figure}
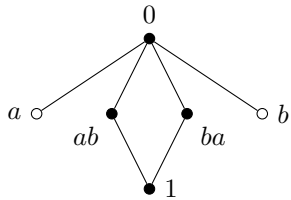
Ironically, our arguments for $S_7$ has its roots in combinatorial developments made specifically for $B_2^1$ as a semigroup in \cite{jac:SAT}, but our arguments for $B_2^1$ as a semiring depend on the $S_7$ adaptation, rather than the original semigroup-theoretic methods presented in \cite{jac:SAT}.  A further consequence of our method is that membership problem for finite algebras in the variety of $S_7$ is \texttt{NP}-complete; this is arguably the simplest example of a finite algebra known with non-polynomial time variety membership (assuming $\texttt{P}\neq \texttt{NP}$).  The result again is somewhat infectious, and extends to the semiring $B_2^1$, again paralleling a result in~\cite{jac:SAT} for $B_2^1$ as a semigroup, though again via $S_7$.
\begin{remark}
The NFB property for the natural semiring structure of~$B_2^1$ has also been independently and contemporaneously established by Mikhail Volkov, using an entirely unrelated approach to that given here: this time relating to the nonfinite axiomatisability of $B_2^1$ as an inverse semigroup \cite{vol21}.
\end{remark}
We then return to the group-theoretic approach of \cite{jac:flat},  which again does not require  the underlying multiplicative semigroup be NFB. The method of \cite{jac:flat} applies only to flat extensions of finite groups, which are exactly the subdirectly irreducible semirings arising from (finite) Clifford semigroups whose natural order as an inverse semigroup produces a semilattice order \cite[\S7]{jacsto:PM}.  We show that, much more generally, any finite ai-semiring containing a nonabelian nilpotent subgroup in its multiplicative reduct is without a finite identity basis.

The structure of the article is as follows.  We begin in Section \ref{sec:SW} with the definition and development of some basic semirings that will be the basis of the main results; aside from $S_7$, we make particular use of semigroups built from finite sets $W$ of commutative words, denoted $S_c(W)$; there is also a monoid version, which we denote by $M_c(W)$.
In Section \ref{sec:hypergraphs} we present the key \emph{hypergraph semiring} construction that we require for our proofs.  This construction is similar to, though not identical to, the construction $S_c(W)$, where $W$ is a set of words corresponding to hyperedges.  Our first main results are contained in Section~\ref{sec:hypersemiringvariety}, where we find that membership of the hypergraph semirings in varieties of our basis semirings is tied to their colourability properties.
Theorem~\ref{thm:hard} shows that $S_7$, $S_c(abb)$ and $B_2^1$ (amongst others) all have \texttt{NP}-hard membership problem for their variety; the case of $S_7$ is the smallest possible size for such an example, and is only the second known example on three elements.
The property of \texttt{NP}-hard membership implies the nonfinite basis property under the assumption of $\texttt{P}\neq \texttt{NP}$, however we provide a more direct, and more general, nonfinite  basis result in Theorem~\ref{thm:p3}.  This theorem demonstrates the FB property for a very broad range of finite ai-semirings, including all but a small class of the form $S_c(W)$, and all of the form $M_c(W)$, for nonempty $W$.  Another example consequence is that a flat semiring with identity is FB if and only if it is the flat extension of a finite group whose nilpotent subgroups are abelian (Corollary~\ref{cor:monoid}).

The specific case of $S_7$  completes the classification of the finite basis property for ai-semirings on at most three elements: 60 of the 61 ai-semirings on three elements are known to be finitely based~\cite{ZRCSD}, while $S_7$ is NFB (see Corollary~\ref{cor:3element}).  Given the prominence of $S_7$ in some of these results, in Section~\ref{sec:S7} we provide some finer analysis of its equational properties, in particular giving a combinatorial description of the equational theory (which has co-\texttt{NP}-complete membership).  In Section~\ref{sec:group} we return to the classification of the FB property for flat groups given by the first author in~\cite{jac:flat}.
Our Corollary~\ref{cor:monoid} already extends this to cover all flat monoid semirings, however in Section~\ref{sec:group} we provide a different extension, by demonstrating that the requirement of being a flat group can be dropped: any finite ai-semiring is NFB if its multiplicative reduct contains a nonabelian Sylow subgroup (Theorem~\ref{thm:pgroup}).
We later demonstrate this result continues to hold in the signature of ai-semirings with~$0$ (Theorem~\ref{thm:pgroup0}).

We conclude the article in Section~\ref{sec:problems} with an extensive list of problems which we feel will provide useful directions to this area.

\section{Semirings}\label{sec:SW}
The main technical construction of hypergraph semirings will given in Section~\ref{sec:hypergraphs}, but in this section we give some further important examples and constructions.
The constructions in this section will all be \emph{flat semirings}: the order is of height one, with the top element equal to a multiplicatively and additively absorbing zero element.  Thus the addition is idempotent and has $x+y:=0$ whenever $x\neq y$.  Notationally, we use $0$ for this absorbing element by default, as it is a multiplicative $0$.  As the top element in the $+$-order however, the notation $\infty$ is a further sensible notation, and we will use this in the rare cases where we wish to discuss semirings with $0$ (where $0$ plays the role of an additive identity and a multiplicative zero).  We also briefly consider semirings with identity, where the constant $1$ will be used to denote a distinguished multiplicative identity element.
Flat algebras have been frequently considered in the literature, including \cite{jac:flat,kunver,sze,wil} for example.  As we prove, the following result follows quickly from \cite[Theorem 3.1(3)]{jac:flat}.
\begin{lem}\label{lem:flatvariety}
The flat semirings generate a variety whose subdirectly irreducible members are precisely the flat semirings.  Within the variety of ai-semirings, this variety can be defined by the equations
\[
x_1ux_2+y_1uy_2+y_1vy_2\approx x_1vx_2+y_1uy_2+y_1vy_2\tag{$*$}\label{eq:flat}
\]
where any of $x_1,x_2,y_1,y_2$ may be empty.
\end{lem}
\begin{proof}
It is immediate consequence of Theorem 3.1(3) of \cite{jac:flat} that the variety generated by flat semirings can be axiomatised by the ai-semiring axioms along with
\[
t(u+v,x_1,\dots,x_n)+s(u,x_1,\dots,x_m)\approx t(u+v,x_1,\dots,x_n)+s(v,x_1,\dots,x_m)\tag{$*'$}\label{eq:flat2}
\]
where $s(x_0,\dots,x_n)$ and $t(x_0,\dots,x_m)$ are any terms, where $x_0$ appears explicitly in $t$.  Up to change of letter names and applications of distributivity of $\cdot$ over $+$, Equation \eqref{eq:flat} in the lemma statement is an example of this, where $s(x_0,x_1,x_2)$ is $x_1x_0x_2$ and $t(y_0,y_1,y_2)$ is $y_1y_0y_2$ (and where we allow $x_1,x_2,y_1,y_2$ to be omitted; Equation \eqref{eq:flat} is really a family of equations).  However any equation of the form of Equation \eqref{eq:flat2} easily follows by repeated applications of Equation \eqref{eq:flat} once distributivity of $\cdot$ over $+$ as been used to rewrite $s$ and $t$ as a sum of multiplicative terms.
\end{proof}
The following lemma explains when  a semigroup with $0$ becomes a flat semiring under the flat semiring addition.
\begin{lem}\label{lem:cancellative}
A semigroup with $0$ becomes a flat semiring if and only if it satisfies the \emph{0-cancellative laws} $xy\approx xz\not\approx 0\rightarrow y\approx z$ and $xy\approx zy\not\approx 0\rightarrow x\approx z$.
\end{lem}
\begin{proof}
We need only check the distributive laws, as $\cdot$ and $+$ are associative, and $+$ is a semilattice operation.  By symmetry we consider distributivity from the left only.  We have $xy+xz\leq x(y+z)$ always, so need to show that $x(y+z)\leq xy+xz$.  If $y=z$ there is nothing to prove, so assume $y\neq z$; then $x(y+z)=0$.   Thus left distributivity is equivalent to the condition that $xy+xz=0$ whenever $y\neq z$. In the flat setting, this is equivalent to either $xy=0$ or $xz=0$ or $xy\neq xz$, which is precisely the contrapositive of the first 0-cancellative law.
\end{proof}
The $0$-direct join of two semigroups $S,T$ with zero element $0$ is the semigroup on the disjoint union $S\backslash\{0\}\mathbin{\dot\cup}T\backslash\{0\}\mathbin{\dot\cup}\{0\}$ where all products within $S$ and $T$ are as before, but products between elements of $S$ with elements of $T$ (in either order) are $0$.  It is clear that if $S$ and $T$ are $0$-cancellative, then their $0$-direct join is $0$-cancellative, so that the following  lemma is a corollary of Lemma \ref{lem:cancellative}
\begin{lem}\label{lem:0directjoin}
The $0$-direct join of two flat semirings is a flat semiring.
\end{lem}
Every group $G$ is cancellative, and becomes a semiring $\flat(G)$ (isomorphic to $G^0$ as a multiplicative semigroup) satisfying the condition of Lemma \ref{lem:cancellative} by adjoining a multiplicative zero element that forms the top element of the $+$-semilattice.  These \emph{flat groups} have previously been studied by the first author in \cite[\S7.7,7.8]{jac:flat} and \cite{jac:eqncomp}.
\begin{pro}\label{pro:monoid}
A finite flat semiring $S$ with multiplicative identity $1$ either contains $S_7$ as a subsemiring, or is the flat extension of a finite group.  This is true in any of the signatures $\{+,\cdot\}$, $\{+,\cdot,0\}$, $\{+,\cdot,1\}$ or $\{+,\cdot,0,1\}$.
\end{pro}
\begin{proof}
Assume that $S$ is a finite flat semiring with multiplicative identity $1\neq 0$.  If $S$ consists of only $1$ and $0$ then it is the flat extension of the trivial group.  Otherwise, assume $S$ contains $a\notin\{1,0\}$. The idempotent power $e$ of $a$ has $ee=e$ and then because $S$ is flat and $e=1e+ee=(1+e)e$ it follows that $e=1$ or $e=0$.  If all elements $a\neq 0$ have idempotent power equal to $1$ then $S$ is the flat extension of a finite group.  Otherwise, there is $a$ with $a\neq 0$ but $a^n=0$ for some $n>1$.  Choose $k\geq 1$ maximum such that $a^k\neq0$.  Then $\{1,a^k,0\}$ is a subsemiring isomorphic to~$S_7$ (in any of the listed signatures).
\end{proof}
We mention that the flat extension of the one-generated free monoid $\{a\}^*$ is a flat semiring with identity that does not contain $S_7$ as a subsemiring.

We now consider a further rich source of examples.  First we note that whenever~$S$ is an ai-semiring with a proper nonempty subset $J$ that is both a multiplicative absorbing ideal and an order-theoretic filter (with respect to $+$), then we may form the \emph{ideal quotient} $S/J$ obtained by identifying all elements in $J$.  The semiring $S_7$ arises as the ideal quotient of $\flat(\{a\}^*)$ for example, using $J=\{a^k\mid k>1\}$.
Now consider the free semigroup $X^+$ on some alphabet $X$.
Let $W$ be a set of words from~$X^+$ and let $W^\leq$ denote the closure in $X^+$ of $W$ under taking (nonempty) subwords.
We define a semigroup on $W^\leq\cup\{0\}$ by letting $0$ be a multiplicative zero element, and for $\bu,\bv\in W^\leq$ letting $\bu\cdot \bv=\bu\bv$ if $\bu\bv\in W^\leq$ and $0$ otherwise.
This construction will be denoted $S(W)$ and is easily seen to be the ideal quotient of $\flat(X^+)$ with respect to $J:=\{\bw\in X^+\mid \bw\notin W^{\leq}\}$.
If we start with $X^*$ (where the empty word is now automaticaly in $W^\leq$), then the corresponding semigroup is a monoid and we use the notation~$M(W)$, though we note that in many previous articles, the notation $S(W)$ was variously used for both the monoid version and the semigroup version.  Following standard notation, we also let $S(\bw)$ abbreviate $S(\{\bw\})$ in the case where $\bw$ is a single word, and similarly for~$M(\bw)$.
There is a commutative variant of this idea.  We let $S_c(W)$ (and~$M_c(W)$) denote the same construction but formed on the free commutative semigroup (or monoid, respectively) generated by $X$.  These again satisfy the conditional cancellativity condition of Lemma \ref{lem:cancellative} so that the following lemma is immediate.
\begin{lem}
For any set of words $W$, the semigroups $S(W)$, $M(W)$, $S_c(W)$ and $M_c(W)$ are flat semirings if the zero element is taken as the top element.
\end{lem}
The semiring $S_7$ arises in this way as  $S_7=M(a)=M_c(a)$, where $a$ is a single letter.  When $a_1,\dots,a_n$ are distinct letters and $n>2$, the semiring $S_c(a_1\dots a_n)$ will be a special example of one of the hypergraph semirings to be constructed in Section~\ref{sec:hypergraphs}, and in fact this particular family of semirings holds an important place in the lattice of semiring varieties generated by semirings of the form $S_c(W)$ and~$M_c(W)$.
\begin{pro}\label{pro:SinM}
For every $n$ and any set of words $W$ containing at least one nonempty word, the semirings $S_c(a_1\dots a_n)$ and $M_c(a_1\dots a_n)$ are contained in the variety generated by $M_c(W)$.
\end{pro}
\begin{proof}
It is clear that $M_c(a)=S_7$ lies in the variety of $M_c(W)$, so it suffices to prove the result in the case that $W=\{a\}$.  We use the notation $1_i^a$ to denote the $n$-tuple consisting of $1$ everywhere except in coordinate $i$ where it equals $a$.  Consider the subsemiring $S$ of $S_7^n$ generated by $\{1_i^a\mid 1\leq i\leq n\}$ and let $J$ denote the set of elements of $S$ containing a $0$ coordinate (to obtain $M_c(a_1\dots a_n)$, add the constant tuple $1$ at this step).  The set $J$ is trivially a multiplicative semigroup ideal because~$0$ is a multiplicative zero.  It is also a filter in the $+$-order because $0$ is the maximum element.  Thus we have a well defined ideal quotient $S/J$.  It is routine to see that  $S/J$ is isomorphic to $S_c(a_1\dots a_n)$ as required.
\end{proof}

We may also prove the following analogue of Proposition~\ref{pro:SinM} for the $S_c(W)$ construction.
\begin{pro}\label{pro:SinS}
Let $W$ be a set of words containing a word $\bw$ that is not a power of a single letter.  Then the semiring $S_c(a_1\dots a_n)$ is contained in the variety generated by $S_c(W)$, for any $n\leq |\bw|$.
\end{pro}
\begin{proof}
Suppose  $\bw$ is a word in $W$ that has length exactly $n$ and that is not a power of a single letter.  It is clear that $S_c(\bw)$ is in the variety of $S_c(W)$, so it suffices to show that $S_c(a_1\dots a_n)$ is in the variety of $S_c(\bw)$.

By commutativity of $\cdot$, we may assume that $\bw$ is written in the form $b_1^{i_i}b_2^{i_2}\dots b_k^{i_k}$ for some distinct letters $b_1,\dots,b_k$ and with $i_1+i_2+\dots+i_k=n$.
There is no loss of generality also in assuming that $i_1\geq i_2\geq \dots\geq i_k$.  Let us denote the set of all $\binom{n}{i_1,\dots,i_k}$ distinct rearrangements of the word $\bw$ as $R$ and for each $i\leq n$ let $\mathsf{a}_i$ denote the $R$-tuple from  $\big(S_c(\bw)\big)^R$ whose position at $r\in R$ is given by the $i^{\rm th}$ position of the rearrangement $r$.
Let $S$ denote the subsemiring of $S_c(\bw)^n$ generated by $\{\mathsf{a}_i\mid i=1,\dots,n\}$.
The product $\mathsf{w}:=\mathsf{a}_1\dots \mathsf{a}_n$ in $S$ has all coordinates equal to $\bw$ (noting that $S_c(\bw)$ is commutative).
We let $J$ denote the set of all elements of $S$ that do not divide $\mathsf{w}$.
This is obviously a multiplicative ideal, but it is also an order-theoretic filter because it contains all tuples with a $0$-coordinate, and the remaining elements of $S$ are incomparable in the $+$-order.
Thus we may form the semiring $S/J$, which we will show is isomorphic to $S_c(a_1\dots a_n)$, under the obvious map defined on generators by $\mathsf{a}_i\mapsto a_i$.
Trivially,  each subset $I\subseteq \{1,\dots,n\}$ has $\prod_{i\in I}\mathsf{a}_{i}\notin J$.  We can use the following claim to show that these are the only products that give rise to elements not in $J$, from which the isomorphism property easily follows.

{\bf Claim 1.} If $\mathsf{a}_{j_1}\dots \mathsf{a}_{j_\ell}$ is a product of generators equal to $\mathsf{w}$, then $\ell=n$ and $\{j_1,\dots,j_\ell\}=\{1,\dots,n\}$.
\begin{proof}[Proof of Claim 1.]
It is trivial that $\ell=n$ is required, as each coordinate of $\mathsf{w}$ is $\bw$, a word of length $n$.  Assume for contradiction that a product $\mathsf{a}_{j_1}\dots \mathsf{a}_{j_n}=\mathsf{w}$ but $\{j_1,\dots,j_n\}\neq\{1,\dots,n\}$.  Thus there is at least one number in $\{1,\dots,n\}$ missing from $\{j_1,\dots,j_n\}$ and at least one number appears twice in the sequence $j_1,\dots,j_n$.  Without loss of generality we may assume that $1\notin \{j_1,\dots,j_n\}$ and $2$ appears at least twice.  By commutativity we may assume that $j_1=j_2=2$.  For each $i\leq k$ let $R_i\subseteq R$ denote those rearrangements that begin with the letter $b_i$ and let $R_{i,j}$ denote the subset of $R_i$ that have second letter equal to $b_j$.  For $j\leq n$, let $m_j$ denote the multiplicity of $j$ in the sequence $j_1,\dots,j_\ell$; so $m_1=0$ and $m_2\geq 2$.

If $i_k=1$ then we are done because on coordinates $r\in R_k$, the tuples $\mathsf{a}_2,\dots,\mathsf{a}_n$ have no occurrences of $b_k$, contradicting $\mathsf{a}_{j_1}\dots \mathsf{a}_{j_n}(r)=\bw$.  Thus we may assume that $i_k\geq 2$.  If $i_k<m_2$ then we have an immediate contradiction, because $\mathsf{a}_{j_1}\dots \mathsf{a}_{j_n}$ has at least $m_2>i_k$ occurrences of $b_k$ on all coordinates in $R_{1,k}$.
If $i_k=m_2$ then we obtain a similar contradiction, because the assumption that $i_k\geq 2$ ensures that for every rearrangement $r$ in $R_{1,k}$ there is a $j>2$ such that the $j^{\rm th}$ letter is $b_k$.
Then if $m_j>0$  there would be more than $m_2+m_j>i_k$ occurrences of $b_k$ in $\mathsf{a}_{j_1}\dots \mathsf{a}_{j_n}(r)$, contradicting  $\mathsf{a}_{j_1}\dots \mathsf{a}_{j_n}(r)=\bw$.  Avoiding this contradiction requires $m_j=0$ for all $j>2$, so that $m_2=n$.  But then $\mathsf{a}_{j_1}\dots \mathsf{a}_{j_n}=\mathsf{a}_2^n$ is not equal to $\bw$ on any coordinate, also a contradiction.

Thus we may assume that $i_k> m_2$.  On each $r\in R_{1,k}$ there are always $i_k-1$ values $j$ from $\{3,4,\dots,n\}$ for which $\mathsf{a}_j(r)=b_k$, however in $\mathsf{a}_{j_1}\dots \mathsf{a}_{j_n}(r)$ we have  $m_2$ occurrences arising from the repetition of $\mathsf{a}_2$, thus only $i_k-m_2$ that are provided by $\mathsf{a}_{j_{m_2}+1}(r),\dots,\mathsf{a}_{j_n}(r)$.
It follows from $m_2>1$ that there is a number $j\in\{3,\dots,n\}$ with $m_j=0$; without loss of generality we may assume $m_3=0$.
Because $m_2=n$ leads immediately contradiction we may further assume without loss of generality that $m_4>0$.
Now consider two rearrangements $r_1,r_2$ in $R_{1,k}$ for which the positions of $b_k$ differ only between positions $3$ and $4$; for example, as $i_1\geq i_2\geq i_k\geq 2$, the rearrangement $r_1$ in $R_{1,k}$ may have $b_1$ in position $3$ and $b_k$ in position $4$, while $r_2$ has $b_k$ in position $3$, and $b_1$ in position $4$, but is otherwise identical.
Then the number of occurrences of $b_k$ in $\mathsf{a}_2^{m_2}\mathsf{a}_{j_{m_2}+1}\dots\mathsf{a}_{j_n}(r_1)$ differs by $m_4$ from the number in $\mathsf{a}_{j_{m_2}+1}\dots\mathsf{a}_{j_n}(r_2)$, contradicting the fact that both equal $\bw$.
\end{proof}
Now we may complete the main proof of the proposition.  We need to show that a product $\mathsf{a}_{j_1}\dots \mathsf{a}_{j_\ell}$ of generators divides $\mathsf{w}$ if and only if $j_1,\dots,j_n$ contains no repeats.  The ``if'' direction is trivial.  The ``only if'' direction follows from  Claim 1, because a concatenation of letters that includes a repeat, cannot extend to one that avoids repeats.  It now follows that $S_c(a_1\dots a_n)$ is isomorphic to $S/J$, under the map determined on generators by $a_i\mapsto \mathsf{a}_i$.
\end{proof}
\begin{remark}
In Proposition \ref{pro:SinS}, the condition that a word needs to consist of more than a single letter is necessary, because
\[
S(a^k)\models x_1\dots x_k\approx x_1\dots x_k+ x_1^k,
\]
which obviously fails on $S(a_1\dots a_k)$ when $k>1$.
\end{remark}
The semirings $S_c(a_1\dots a_n)$ (for $n\geq 3$) will play an important role in the article.   As already noted, the semiring $S_7$ coincides with $M_c(a)$, while we later find a parallel role for the semiring $S_c(abb)$.
\begin{lem}\label{lem:incomp}
For $k>2$, neither the semiring $S_7$ nor the semiring $S_c(ab^{k-1})$ lies in the variety generated by the other.
\end{lem}
\begin{proof}
We have $S_7\models x^2\approx x^3$, while $S_c(ab^{k-1})$ fails this when $x$ is assigned $b$.  Conversely, $S_c(ab^{k-1})$ is $(k+1)$-nilpotent (satisfying $x_1\dots x_{k+1}\approx y_1\dots y_{k+1})$, which fails on $S_7$ when the $x_i$ are assigned $1$ while the $y_i$ are assigned $0$.
\end{proof}
In contrast, Propositions \ref{pro:SinM} and~\ref{pro:SinS} together immediately imply the following.
\begin{lem} \label{lem:a1tok}
For $k>1$ we have $\mathsf{V}(S_c(a_1\dots a_{k}))\subseteq \mathsf{V}(S_7)\wedge \mathsf{V}(S_c(ab^{k-1}))$.
\end{lem}

The (multiplicative) semigroups, both $S_c(W)$ and $M_c(W)$ are commutative, so are finitely based as semigroups, Perkins \cite{per}.  Similarly, the underlying semigroup of $\flat(G)$ is the Clifford semigroup $G^0$, which is finitely based as a semigroup (provided $G$ is finite), due primarily due to the fact that every finite group has a finite basis for its equations, Oates and Powell \cite{oatpow}.  The semigroups $S(W)$ are also finitely based as semigroups (assuming $W$ is finite) as they are nilpotent, and every term is equivalent to one of length at most one more than the longest word in $W$.  The semigroups $M(W)$ are variously finitely based and nonfinitely based, and since the original work of Perkins \cite{per}, have been amongst the richest source of challenging behaviour for the finite basis property and related problems in finite semigroups; see the work of Olga Sapir \cite{osap}, Jackson and Sapir \cite{jacsap}, amongst many others.  Following the  notation from~\cite{jacvol}, we say that an algebra $A$ (or variety $V$) satisfying a property $P$ is \emph{inherently nonfinitely based} (INFB) \emph{relative to $P$}  if every variety satisfying $P$ and containing $A$ (or $V$, respectively) is without a finite basis for its equations.  When $P$ is ``generates a locally finite variety'', then we omit reference to $P$ and say that $A$ (or $V$) is \emph{inherently nonfinitely based} (INFB).  It is a trivial consequence of the work of Mark Sapir \cite{sap2}, that no semigroup $M(W)$ (for finite~$W$) is INFB. \ The following folklore result appears to be well known to researchers worldwide but does not appear to have been stated in published form in English.
\begin{thm}\label{thm:INFB}
\begin{enumerate}
\item Let $S=\langle S;+,\cdot\rangle$ be an ai-semiring whose multiplicative reduct $\langle S;\cdot\rangle$ is locally finite.  Then $S$ is locally finite.
\item Let $\Sigma$ be a finite set of semigroup identities defining a locally finite variety of semigroups.  Then, in conjunction with the axioms for ai-semirings, the identities $\Sigma$ define a locally finite variety of ai-semirings.
\item The multiplicative reduct of an ai-semiring $S$ is INFB if $S$ is INFB.
\end{enumerate}
\end{thm}
\begin{proof}
For (1), let $A$ be a finite subset of $S$.  The subsemigroup $S_{A,(\cdot)}$ of $\langle S;\cdot\rangle$ generated by $A$ is finite, but by distributivity and idempotence, every element of the subsemiring $S_{A,(+,\cdot)}$ of $S$ generated by $A$ can be written as a sum of elements of~$S_{A,(\cdot)}$, without repeats, showing that $|S_{A,(+,\cdot)}|\leq 2^{|S_{A,(\cdot)}|}$.

Item (2) follows from (1) given that all semirings satisfying $\Sigma$ have locally finite semigroup reducts, and hence are locally finite.  Item (3) is an immediate corollary of (2).
\end{proof}
\begin{remark}
Theorem \ref{thm:INFB} admits many generalisations and variations.  The proof used only the fact that arbitrary elements of $S_{A,(+,\cdot)}$ can be written as sums of elements of $S_{A,(\cdot)}$.  This requires only left and right distributivity of $\cdot$ over $+$, and enough properties of $+$ to ensure that that sums over a finite alphabet are equivalent sums of bounded length.   As an example, the theorem generalises to semirings satisfying only a periodic law for addition, as the variety of commutativity $+$-semigroups of any given index and period is locally finite.  These ideas generalise to larger signatures, subject to the availability of suitable distributivity properties.  Something very similar also occurs, for example, in the class of (one-sided) restriction semigroups, see \cite[Theorem 9.2]{jon}.  There is also no significant generating power obtained by adding (finitely many) constants, so that semirings with $0$ also satisfy a version of the theorem.
\end{remark}
\section{Hypergraphs and hypergraph semirings}\label{sec:hypergraphs}
Our primary construction is a reimagining of some hypergraph methods developed for semigroups in~\cite{jac:SAT} and then subsequently in \cite{jaczha}, however the methods in the present article are a parallel rather than direct application.  In \cite{jac:SAT,jaczha} the constructions depend in a fundamental way on the structure of the non-commutative semigroup~$B_2^1$:  hyperedge encodings are kept separate in long products by way of a noncommuting barrier.  In contrast, the commutativity of $\cdot$ in~$S_7$ and~$S_c(abc)$ appears to play a central role in the arguments of the present article, and the hyperedge encodings are instead gathered together by long sums. 

We begin with some basic background on hypergraphs and related concepts; there is some overlap with the articles \cite{hamjac:hyper} and \cite{jaczha}, as well as the manuscript~\cite{jac:SAT}, however we are not able to present our semiring-specific developments without presenting them here again.

We adopt standard notation and terminology, as in \cite{hamjac:hyper}.  A \emph{hypergraph} is a pair $H=(V,E)$, where $E$ is a family of nonempty subsets of $V$.  If $|e|\leq k$ for each $e\in E$, we say that $H$ is a \emph{$k$-hypergraph}.  If $|e|=k$ for all $e\in E$ then $H$ is a \emph{$k$-uniform hypergraph}.  We also view $k$-hypergraphs as relational structures in the signature of a single $k$-ary relation consisting of the $k$-tuples $\{(v_1,\dots,v_k)\mid \{v_1,\dots,v_k\}\in E\}$; we refer to a $k$-hypergraph as a \emph{$k$-hypergraph structure} when in this context.  So,  for a hyperedge $\{u,v\}$ in a $3$-hypergraph will, in the corresponding $3$-hypergraph structure, give rise to six tuples $uuv,uvu,vuu,uvv,vuv,vuu$; here we adopt a standard convention of abbreviating tuples $(x,y,z)$ as strings $xyz$.  Similarly, any relational structure in the signature of a single $k$-ary relation $R$ becomes a $k$-hypergraph, by taking the hyperedges $\{v_1,\dots,v_k\}$ for each tuple $(v_1,\dots,v_k)\in R$.  The $k$-hypergraph structure formed from a general $k$-relational structure is a kind of ``set-closure'' of the relations: a tuple is present in this closure if it contains the same elements as some other tuple in the relation.

A \emph{path} in a hypergraph is a sequence $v_0,e_0,v_1,e_1,\dots,v_{n-1},e_{n-1}$ alternating between vertices and hyperedges, with no repeats, and such that $v_0\in e_0$ and $v_{i+1}\in e_i\cap e_{i+1}$ (with addition in the subscript modulo $n$).  A path is a \emph{cycle} if we also have $v_0\in e_{n-1}\cap e_0$; we often write $v_0,e_0,v_1,e_1,\dots,v_{n-1},e_{n-1},v_0$.
The \emph{girth} of hypergraph is the length of the shortest cycle; hypergraphs without cycles are called \emph{hyperforests} and are conventionally set to have infinite girth.

We will also make use of the $2$-element structure $\mathbbold{2}=\langle \{0,1\};R\rangle$, where $R$ is the ternary relation $\{110,101,011\}$.  This is not quite a $3$-hypergraph structure, but it is similar, and of the same signature.  The homomorphism problem for $\mathbbold{2}$ is the positive 2-in-3SAT problem and is well known to be \texttt{NP}-complete.  This computational problem is more frequently encountered as \emph{positive 1-in-3SAT}, as in~\cite{jac:SAT}  for example, however we are switching the role of $0$ and $1$ in the present article, as this notation better matches the semiring $S_7$.

The positive 2-in-3SAT problem corresponds to a specific kind of 3-hypergraph colouring problem: a function from $V\to\{0,1\}$ in which each 3-hyperedge is sent to one of the tuples $\{110,101,011\}$.  This extends further to $k$-hypergraphs: a \emph{$(k-1)$-in-$k$ satisfaction} of a $k$-hypergraph is a function $V\to\{0,1\}$ mapping each hyperedge to $\{\overbrace{1\dots11}^{k-1}0,1\dots 101,\dots,011\dots 1\}$: that is, all but one vertex in each hyperedge is given the value $1$.  The more general and more familiar notion of $\ell$-colouring of a $k$-hypergraph $\mathbb{H}=(V,E)$ is a function $\gamma$ from $V$ to $(\{0,\dots,\ell-1\},F)$ where $F$ consists of all subsets of $\{0,\dots,\ell-1\}$ that have cardinality between $2$ and~$k$ (inclusive) and such that $\gamma(e)\in F$ for all $e\in E$.  (This definition carries across equivalently to the setting of $k$-hypergraph structures; see \cite[\S2]{hamjac:hyper}.)  The hypergraph $2$-colourability of $3$-hypergraphs structures can be seen to be just  the well known \emph{positive NAE3SAT} problem.  One final notion is required.  A $k$-hypergraph is \emph{${\leq}i$-robustly $(k-1)$-in-$k$ satisfiable} if every valid partial $(k-1)$-in-$k$ satisfaction of any set of at most $i$ vertices extends to a full $(k-1)$-in-$k$-satisfaction (see \cite{AGK,ham,hamjac,jac:SAT}).  We will only use the case of $i=2$ here, so give some further details on what the concept  means in this particular case.  Any mapping from vertices $x,y$ to one of the pairs $(1,1),(0,1),(1,0)$ is a valid partial satisfaction, while $(0,0)$ is valid only if $\{x,y\}$ is \emph{not} a subset of a hyperedge: if $\{x,y\}$ was a subset of hyperedge $e$ then it is trivial that the partial assignment $x\mapsto 0$ and $y\mapsto 0$ cannot extend to a $(k-1)$-in-$k$ satisfying assignment, as any such assignment would give exactly one vertex in $e$ the value $0$.

Our hardness and nonfinite axiomatisability results will hinge on two convenient facts relating to $k$-hypergraph structures.  The first item is proved by Jackson \cite[Theorem 6.2]{jac:SAT}.  The second item is by Erd\H{o}s and Hajnal \cite{erdhaj}; see Theorems 2.6 and 2.7 of \cite{hamjac:hyper} for further discussion.
\begin{thm}\label{thm:facts}
\begin{enumerate}
\item  The following promise-problem is \texttt{NP}-hard for any fixed $k\geq 3$.  Given a $k$-uniform hypergraph $H=(V,E)$ with girth at least $5$\up:
\begin{itemize}
\item[Yes:] $H$ is ${\leq}2$-robustly 2-in-3 satisfiable.
\item[No:] $H$ is not 2-in-3 satisfiable.
\end{itemize}
\item For every $n$ and every $\ell\geq 2$ there exists a $k$-uniform hypergraph $H_{n,\ell}$ that is not $\ell$-colourable, but every $n$-element substructure is a $k$-uniform hyperforest.
\end{enumerate}
\end{thm}
We note that unfortunately it is not possible to extend the \texttt{NP}-hardness of the promise problem in (1) to the more extreme possibilities of (2): it is known that there is a polynomial time algorithm to distinguish $2$-in-$3$-colourable hypergraphs (the YES instances of positive 2-in-3SAT) from those that are not $2$-colourable (the NO instances of NAE3SAT); see \cite[Example 7.15 and Theorem 7.19]{bragur} or \cite[\S8]{BKO}.

The next lemma will be a $k$-hypergraph generalisation of \cite[Lemma 6.3]{jac:SAT}; the arguments are also quite similar, though we give them for completeness.
Let us say that a set of vertices $\{v_1,\dots,v_\ell\}$ is a \emph{subhyperedge} for some $k$-hypergraph $\mathbb{H}$ if it is a subset of a hyperedge of $\mathbb{H}$. Two subhyperedges of size $k-1$ will be said to be \emph{linked}, if there is a single vertex $w$ that completes both of the subhyperedges to full hyperedges of $\mathbb{H}$.  The \emph{link graph} of $\mathbb{H}$ is the graph whose vertices are the $(k-1)$-element sets of vertices, with two such sets adjacent if they are linked.  We make only marginal use of the link graph, though it is useful conceptual tool in understanding the consequences of some of the defining properties of our semiring construction later.
\begin{lem}\label{lem:hyperprop2}
Let $k>2$.  If $\mathbb{H}$ is a $k$-uniform hypergraph of girth at least $4$, then the following are true.
\begin{itemize}
\item[(I)] No two distinct hyperedges share more than one common vertex.
\item[(II)] If $\ell>2$ and every proper subset of $\{v_1,\dots,v_\ell\}$ is a subhyperedge, then $\{v_1,\dots,v_\ell\}$ is a subhyperedge; in particular, $\ell\leq k$.
\item[(III)] For all $\ell>1$, a set $\{v_1,\dots,v_\ell\}$ is a subhyperedge if and only if all $2$-element subsets of $\{v_1,\dots,v_\ell\}$ are subhyperedges.
\item[(IV)] The link graph of $\mathbb{H}$ consists of a disjoint union of cliques.
\end{itemize}
\end{lem}
\begin{proof}
Item (I) is equivalent to the absence of $2$-cycles in a hypergraph: if $\{u,v\}$ is a subset of two distinct hyperedges $e,f$ then $u,e,v,f,u$ is a $2$-cycle.
For item~(II), assume that every proper subset of $\{v_1,\dots,v_\ell\}$ is a subhyperedge.  If $\ell>3$ then we can use (I) to deduce that $\{v_1,\dots,v_\ell\}$ is a subhyperedge as follows.  For each $(\ell-1)$-element subset $S$---a subhyperedge by assumption---identify a hyperedge $e_S$ containing $S$. As distinct $(\ell-1)$-element subsets $S,T$ of $\{v_1,\dots,v_\ell\}$ share $\ell-2\geq 2$ elements, it follows from (I) that $e_S=e_T$ so that there was in fact a single hyperedge $e$ containing every $\ell-1$ element subset of $\{v_1,\dots,v_\ell\}$.  Thus $\{v_1,\dots,v_\ell\}\subseteq e$.  For the case that $\ell=3$ use the fact that if all $2$-element subsets of $\{v_1,v_2,v_3\}$ are subhyperedges, then we have hyperedges $e_1$ containing $\{v_1,v_2\}$, $e_2$ containing $\{v_2,v_3\}$ and $e_3$ containing $\{v_1,v_3\}$.  As there are no $3$-cycles by assumption, the sequence $v_1,e_1,v_2,e_2,v_3,e_3,v_1$ contains a repeat $e_i=e_j$, so that there is a hyperedge containing $\{v_1,v_2,v_3\}$.

Now for (III).  The forward implication is trivial, so let assume now that $\{v_1,\dots,v_\ell\}$ is not a subhyperedge. If there is a subset of size at least $2$ that is not a subhyperedge we may solve the problem for this and obtain the desired solution for $\{v_1,\dots,v_\ell\}$.  Thus there is no loss of generality in assuming that either $\ell=2$ or every subset of $\{v_1,\dots,v_\ell\}$ is a subhyperedge.  But then (II) implies that $\ell=2$ and we are done.

If we let $\sim$ denote the edge relation of the link graph, then property (IV) is equivalent to
\[
A\sim B\sim C\rightarrow (A\sim C\vee A=C)
\]
 for link graph vertices $A,B,C$.  Assume that $A=\{u_1,\dots,u_{k-1}\}$ is linked to $B=\{v_1,\dots,v_{k-1}\}$ by way of some vertex $v$ and that $B$ is linked to $C=\{w_1,\dots,w_{k-1}\}$ by way of some vertex $w$.  But then $v=w$ as $\{v_1,\dots,v_{k-1},v\}$ and $\{v_1,\dots,v_{k-1},w\}$ overlap by more than $1$ vertex so that either
 \[
 C=\{w_1,\dots,w_{k-1},w\}=\{u_1,\dots,u_{k-1},w\}=A
 \]
 or that $A$ links to $C$, as required.
\end{proof}

The following lemma is essentially  a minor modification of Lemmas 2.8 and~2.9 of \cite{hamjac:hyper}, and the proof is included for completeness only.  We use the familiar (and easy to prove) property that every finite hyperforest either contains no hyperedges at all, or it contains a leaf: a hyperedge that is has just at most one vertex in common with any other hyperedge.
\begin{lem}\label{lem:forest}
A $k$-uniform hyperforest $\mathbb{F}$ is ${\leq}2$-robustly $(k-1)$-in-$k$ satisfiable.
\end{lem}
\begin{proof}
This is by induction on the number of hyperedges of $\mathbb{F}$.  For a single hyperedge the statement is trivial.  Now assume that it is true for $k$-uniform hyperforests consisting of $n$ hyperedges, and consider a $k$-uniform hyperforest $\mathbb{F}$ of $n+1$ hyperedges.  Identify a leaf $e$ of $\mathbb{F}$. If $e$ is an isolated hyperedge then the statement follows immediately, so we assume that $e$ does share one vertex, $v$ with the remaining part of $\mathbb{F}$.  Let $\mathbb{F}_e$ be the result of removing the hyperedge $e$ and all of its vertices except $v$.  Every $(k-1)$-in-$k$ satisfaction of $\mathbb{F}_e$ extends  to a $(k-1)$-in-$k$ satisfaction of $\mathbb{F}$: if $v$ is coloured $0$ then this colouring is unique, and if $v$ is coloured $1$ then there are $k-1$ choices of a vertex in $e\backslash\{v\}$ to colour $0$ and the rest are coloured $1$.  It is now easy to use the assumed ${\leq}2$-robust $(k-1)$-in-$k$ satisfiability of $\mathbb{F}_e$ to prove the same for $\mathbb{F}$.  For pairs of vertices $x,y$ that lie in $\mathbb{F}_e$ the required property holds by induction.  If exactly one is  in $e$ (say, $x$) and neither are $v$, then we may fix any colouring of $\mathbb{F}_e$ that achieves the desired colour for $x$ and then use the free extendability to $e\backslash\{v\}$ just observed to give either of $0$ or $1$ to $y$.  If both are in $e$ then we can $(k-1)$-in-$k$ satisfy $e$ to achieve any colouring of $x,y$ except the prohibited $(x,y)\mapsto (0,0)$ and then extend this to a $(k-1)$-in-$k$ satisfaction of the rest of $\mathbb{F}$ by choosing any colouring of $\mathbb{F}_e$ that takes the already determined colour for $v$ (which can be done by induction).
\end{proof}
Let us consider, for some $k>2$, a $k$-uniform hypergraph $\mathbb{H}$ of girth $m\geq 4$, noting that each of the properties in Lemma \ref{lem:hyperprop2} will hold for $\mathbb{H}$; we further assume that $\mathbb{H}$ contains no isolated vertices, as such vertices play no role in the colourability properties (nor in $(k-1)$-in-$k$ satisfiability properties) of $\mathbb{H}$.  We now construct an ai-semiring $S_\mathbb{H}$ from $\mathbb{H}$; collectively we refer to these as \emph{hypergaph semirings}, though note that these are entirely different from the hypergraph algebras and flat hypergraph algebras of \cite{kunver}, which are rarely semirings.  In the case that $k=3$, the construction is in surprisingly close analogy to the semigroup $S_I$ of \cite[\S6]{jac:SAT}\footnote{The similarity is  because both encode the hypergraph hyperedges by way of products, but the  algebraic scaffolding that encodes the broader homomorphism and colouring properties of hypergraphs are perhaps less similar: one uses the structure of the main regular $\mathscr{D}$-class of $B_2^1$, while the other depends fundamentally on the flat semilattice structure of $S_7$ and related semirings.} and the present ideas loosely follow those in \cite{jac:SAT}, as much as this is possible to make sense, given the different settings.  Only $k=3$ is needed to prove the nonfinite basis property for~$S_7$, but the move to arbitrary $k>2$ seems required for stronger properties.

 The generators of $S_\mathbb{H}$ consist of the following elements:
\begin{enumerate}
\item[(i)] $0$, a multiplicative and additive zero element;
\item[(ii)] for each vertex $v$ of $\mathbb{H}$ an element $\ba_v$.
\end{enumerate}
The multiplicative part of the semiring is subject to the following rules:
\begin{enumerate}
\item $\ba_u\ba_v=0$ if $u=v$ or $\{u,v\}$ is not a subhyperedge.
\item $\ba_u\ba_v=\ba_v\ba_u$.
\item $\ba_{u_1}\dots \ba_{u_k}=:\ba$ whenever $\{u_1,\dots,u_k\}$ is a hyperedge of $\mathbb{H}$.
\item $\ba_{u_1}\dots \ba_{u_{k-1}}=\ba_{v_1}\dots \ba_{v_{k-1}}$ if $\{u_1,\dots,u_{k-1}\}$ and $\{v_1,\dots,v_{k-1}\}$ are linked $(k-1)$-tuples.
\end{enumerate}
When  $\{u_1,\dots,u_{k-1}\}$ is a $(k-1)$-element subhyperedge, then we let $[u_1\dots u_{k-1}]$ denote the set of all $(k-1)$-sets $\{x_1,\dots,x_{k-1}\}$ such that $\{u_1,\dots,u_{k-1}\}$ is linked to $\{x_1,\dots,x_{k-1}\}$; so $[u_1\dots u_{k-1}]=[x_1\dots x_{k-1}]$ in this case.   We write $\ba_{[u_1\dots u_{k-1}]}$ to denote the value $\ba_{u_1}\ba_{u_2}\dots \ba_{u_{k-1}}$.  Rule (4) shows that when $\{u_1,\dots,u_{k-1}\}$ is linked to $\{x_1,\dots,x_{k-1}\}$ we have
\[
\ba_{[u_1\dots u_{k-1}]}=\ba_{u_1}\ba_{u_2}\dots \ba_{u_{k-1}}=\ba_{x_1}\ba_{x_2}\dots \ba_{x_{k-1}}=\ba_{[x_1\dots x_{k-1}]},
\]
 which is consistent with the fact that  $[u_1\dots u_{k-1}]=[x_1\dots x_{k-1}]$.  We note that item~(IV) of Lemma \ref{lem:hyperprop2} shows that we do not expect significant flow-on consequences of Rule (4); this is established in the course of the proofs anyway, but noting it may be useful to help establish intuition about the equalities defining $S_\mathbb{H}$.  We mention also that Rule (4) is essentially a special instance of the 0-cancellativity condition in Lemma \ref{lem:cancellative} and so is required for a construction to be a flat semiring.

Additively, we let $S_\mathbb{H}$ be flat ai-semiring: a height $1$ semilattice, with $0$ being an absorbing element for addition.  We also mention that it is possible to adjoin a multiplicative identity element $1$ to this semiring (also set to be order-theoretically incomparable to all other elements except $0$), and we denote this by $M_\mathbb{H}$.

Rules (1)--(4) (and the flat semiring property) are a semiring presentation, so $S_\mathbb{H}$ is certainly a semiring.  The following lemma shows that there are ``no surprises'' to the structure of $S_\mathbb{H}$: the obvious equalities in the rules defining $S_\mathbb{H}$ are the only ones.  This lemma is not used elsewhere in the article, and is given  only in order to aid  intuition of the construction.
\begin{lem}\label{lem:clarification}
Let $\mathbb{H}$ be a $k$-uniform hypergraph of girth at least $4$.
Every non-zero element in $S_\mathbb{H}$ is of the form $\ba_{u_1}\dots \ba_{u_\ell}$ for some subhyperedge $\{u_1,\dots,u_\ell\}$ of size $\ell\leq k$.  Two such products $\ba_{u_1}\dots \ba_{u_i}$ and $\ba_{v_1}\dots \ba_{v_j}$ are distinct if they are distinct subhyperedges, unless
\begin{itemize}
\item $i=j=k$, so that both $\{u_1,\dots ,u_k\}$ and $\{v_1,\dots ,v_k\}$ are hyperedges\up{:} in this case $\ba_{u_1}\dots \ba_{u_k}=\ba_{v_1}\dots \ba_{v_k}$\up{;} or
\item $i=j=k-1$ and $\{u_1,\dots ,u_{k-1}\}$ and $\{v_1,\dots ,v_{k-1}\}$ are linked\up{:} in this case $\ba_{u_1}\dots \ba_{u_{k-1}}=\ba_{v_1}\dots \ba_{v_{k-1}}$.
\end{itemize}
\end{lem}
\begin{proof}
The flat semilattice property ensures  every nonzero element can be written as a product of generators $\ba_{u_1}\dots \ba_{u_\ell}$.  (If we were to consider $M_\mathbb{H}$ then we would also have trivial extensions of this product involving $1$).  Lemma \ref{lem:hyperprop2}(III) and applications of Rules (1)--(2) then easily show that $\{u_1,\dots,u_\ell\}$ is a subhyperedge and that there are no repeats in $u_1,\dots,u_\ell$, so that $|\{u_1,\dots,u_\ell\}|=\ell\leq k$.  It now remains to show that except in the two itemised cases there are no further identifications.  We give sketch details only, reiterating that it is a consequence of later arguments (not depending on the present lemma).  Each of the Rules (1)--(4) either return $0$ or do not change the length of a product.  Thus two nonzero products $\ba_{u_1}\dots \ba_{u_i}$ and $\ba_{v_1}\dots \ba_{v_j}$ must have the same length $i=j=:\ell$.  Clearly Rule (2) only enables rearrangements up to commutativity.  If $\ell<k-1$ then this is the only rule that applies.   If $\ell=k$ then applications of Rule (4) apply universally to show equality of all nonzero products of length $k$, so there is nothing to prove.  When $\ell=k-1$, applications of Rule (3) does not yield further equivalences beyond those covered directly in the rule itself because of Lemma \ref{lem:hyperprop2}(IV).
\end{proof}	

\section{Hypergraph semirings in varieties}\label{sec:hypersemiringvariety}
The hypergraph semirings $S_\mathbb{H}$ encode hypergraphs in a transparent way.  In this section we find how membership of these hypergraph semirings $S_\mathbb{H}$ in varieties will relate to various colouring conditions of the hypergraph $\mathbb{H}$, provided that the varieties containing some of the critical objects we encountered in Section \ref{sec:SW}.

\begin{lem}\label{lem:HinS7}
If $\mathbb{H}$ is ${\leq 2}$-robustly $(k-1)$-in-$k$ satisfiable, then $S_\mathbb{H}\in \mathsf{V}(S_7)\wedge \mathsf{V}(S_c(ab^{k-1}))$.
\end{lem}
\begin{proof}
We let $S$ denote either of $S_7$ or $S_c(ab^{k-1})$.  At many steps of the proof it will make no difference as to which is chosen, but when there is a difference, we give details for the $S_7$ case, with the difference required for $S_c(ab^{k-1})$ given in brackets.  In the case of $S=S_7$, we also invite the reader to consider the situation where the monoid construction $M_\mathbb{H}$ is used in place of $S_\mathbb{H}$: the arguments below carry through with trivial amendment.

Let  $T$ denote the set of all valid $(k-1)$-in-$k$ satisfactions of the vertices of~$\mathbb{H}$.  For each vertex $u$, define a tuple $a_u$ in $S^T$ by
\[
a_u(\phi)=\begin{cases}
a &\text{ if $\phi(u)=0$}\\
1 &\text{  if $\phi(u)=1$.}
\end{cases}
\]
(When $S=S_c(ab^{k-1})$, the value $a_u(\phi)=1$ is to be replaced by $a_u(\phi)=b$.)
Let~$A$ denote the subsemiring of $S^T$ generated by the elements of the form $a_u$.  We claim that $S_\mathbb{H}$ can be obtained from $A$ by factoring out the ideal $J$ generated by all elements of $A$ having a coordinate equal to $0$.  This is trivially an upset of the join semilattice structure on $A$, and an ideal of the multiplicative structure of $A$, so is a well-defined quotient.  We now wish to show that this quotient is isomorphic to $S_\mathbb{H}$, which will show that $S_\mathbb{H}\in\mathsf{V}(S)$.  The approach is as follows.  We first show that each of the defining equations (1)--(4)  of $S_\mathbb{H}$ holds, with $a_u$ replacing $\ba_u$, and that $A/J$ is a flat semilattice.   This shows that $A/J$ is a quotient of $S_\mathbb{H}$.  To show that it is isomorphic, we show that all products of the generators $a_u$ that are not obviously equivalent due to (1)--(4) are distinct in $A/J$.

We begin by checking the  defining laws (1)--(4) of $S_\mathbb{H}$ hold, with $a_u$ replacing $\ba_u$, and noting that we do have the required generators: a $0$ element (namely $J$) and an element $a_u$ for each vertex $u$.
If $u=v$ or $\{u,v\}$ does not extend to a hyperedge, then by ${\leq 2}$-robust $(k-1)$-in-$k$ satisfiability, there is a $(k-1)$-in-$k$ satisfaction that gives both $u$ and $v$ the value $0$.  Then $a_ua_v(\phi)=aa=0$ so that $a_ua_v\in J$ showing that (1) holds.  The commutativity property (2) holds trivially, based on the commutativity of $S$.  For item (3), consider any hyperedge $\{u_1,\dots,u_k\}$ and note that for any $(k-1)$-in-$k$ satisfaction $\phi$ we have that $\phi$ assigns exactly one of $u_1,\dots,u_k$ the value $0$, and the remainder are given the value $1$.
Up to commutativity, this yields $a_{u_1}\dots a_{u_k}(\phi)=a\overbrace{11\dots 1}^{k-1}=a$ in the case of $S=S_7$ and $a\overbrace{bb\dots b}^{k-1}=ab^{k-1}$ in the case of $S=S_c(ab^{k-1})$.  Thus $a_{u_1}\dots a_{u_k}$ is the constant tuple $\underline{a}$, showing that (3) holds.
(In the case of $S=S_c(ab^{k-1})$ it is the constant tuple $\underline{ab^{k-1}}$.)
Finally, consider when $\{u_1,\dots,u_{k-1}\}$ and $\{v_1,\dots,v_{k-1}\}$ are linked by way of some vertex~$w$.
For any $(k-1)$-in-$k$ satisfaction $\phi$, if $\phi(w)=0$ then $\phi(u_1)=\dots=\phi(u_{k-1})=\phi(v_1)=\dots=\phi(v_{k-1})=1$ so that $a_{u_1}\dots a_{u_{k-1}}(\phi)=a_{v_1}\dots a_{v_{k-1}}(\phi)=1$ (or $b^{k-1}$ in the case of $S=S_c(ab^{k-1})$).
Alternatively, if $\phi(w)=1$, then precisely one of $u_1,\dots,u_{k-1}$ and one of $v_1,\dots,v_{k-1}$ is $0$ and the other are $1$; and this case $a_{u_1}\dots a_{u_{k-1}}(\phi)=a_{v_1}\dots a_{v_{k-1}}(\phi)=a$ (or $ab^{k-2}$ in the case of $S_c(ab^{k-1})$).
Either way, we have $a_{u_1}\dots a_{u_{k-1}}(\phi)=a_{v_1}\dots a_{v_{k-1}}(\phi)$ for all $\phi$ so that equality (4) holds.

The last remaining property of $S_\mathbb{H}$ that needs checked for $A/J$ is that it is a flat semilattice.  This is immediate, given that $S$ is a flat semiring and each pair of distinct elements differ in some coordinate.  The sum of two such elements will then take the value $0$ on that coordinate and hence lie in $J$.

In order to complete the proof that $A/J\cong S_\mathbb{H}$ we must show that no further equalities hold in $A/J$ beyond those that are forced by the equalities (1)--(4).  At this point we still do not know if there are some subtle equalities that are consequences of (1)--(4), but in the course of the proof we will discover that only the obvious ones hold in $A/J$, which will imply that $A/J\cong S_\mathbb{H}$.

The elements in $A/J$ can all be written as either the set $J$ (a multiplicative $0$) or as a product of the form $a_{u_1}\dots a_{u_\ell}$ for some $\ell$.  Thus, by the ${\leq 2}$-robust $(k-1)$-in-$k$ satisfiability property, all elements of $A$ have a coordinate equal to either $a$ or $0$.  We can deduce that if the product $a_{u_1}\dots a_{u_\ell}$ in $A$ is not an element of $J$, then it contains no repeats, as we noted commutativity and the property $a_ua_u\in J$ already.   We may also assume that $\{u_1,\dots,u_\ell\}$ is a subhyperedge: by Lemma~\ref{lem:hyperprop2}(III) we know that otherwise there is a pair $u,v\in \{u_1,\dots,u_\ell\}$ such that $\{u,v\}$ is not a subhyperedge, and hence the known equalities (1) and (2) show that the product lies in $J$.  This also implies that $\ell\leq k$, as no set of size more than $k$ can extend to a $k$-element hyperedge.

Now let $\{u_1,\dots,u_{i}\}$ and $\{v_1,\dots,v_{j}\}$ be distinct subhyperedges of size $i$ and~$j$ respectively.
Our goal is to show that if $a_{u_1}\dots a_{u_i}=a_{v_1}\dots a_{v_j}$ then either $i=j=k$ or $i=j=k-1$ and $\{u_1,\dots,u_{i}\}$ and $\{v_1,\dots,v_{j}\}$ are linked.
We prove the contrapositive.
First, if $i=k$ and $j<k$ then we know that $a_{u_1}\dots a_{u_i}$ is the constant tuple $\underline{a}$ (or the constant tuple $\underline{ab^{k-1}}$ when $S=S_c(ab^{k-1})$).  But there is a hyperedge $\{v_1,\dots,v_k\}$ extending $\{v_1,\dots,v_{j}\}$ and we may find a $(k-1)$-in-$k$ satisfaction $\phi$ of $\mathbb{H}$ that gives $v_k$ the value $0$.
Then $a_{v_1}\dots a_{v_j}(\phi)=1$ so that  $a_{u_1}\dots a_{u_i}(\phi)=a\neq 1=a_{v_1}\dots a_{v_j}(\phi)$.
(In the case of $S=S_c(ab^{k-1})$ we have $a_{v_1}\dots a_{v_j}(\phi)=b^j$, while $a_{v_1}\dots a_{v_j}(\phi)=ab^{j-1}$.)
So now assume that both $i,j\leq k-1$ and that $\{u_1,\dots,u_{i}\}$ is not linked to $\{v_1,\dots,v_{j}\}$.
We now find a $(k-1)$-in-$k$ satisfaction~$\phi$ such that one of $\{\phi(u_1),\dots,\phi(u_{i})\}$ and $\{\phi(v_1),\dots,\phi(v_{j})\}$ is $\{1\}$, and the other contains $0$, which will complete the proof as then one of $a_{u_1}\dots a_{u_i}(\phi)$ and $a_{v_1}\dots a_{v_j}(\phi)$ is $1$
(or a power of $b$ in the case of $S=S_c(ab^{k-1})$) and the other is~$a$ (or contains an $a$).

For convenience, let $u_{i+1},\dots,u_k$ be such that $\{u_1,\dots,u_k\}$ is a hyperedge and $v_{j+1},\dots,v_k$ be such that $\{v_1,\dots,v_k\}$ is a hyperedge.  Without loss of generality, we may assume that $i\geq j$ and that $u_1\notin \{v_1,\dots,v_j\}$.
It's possible that $\{u_1,\dots,u_k\}=\{v_1,\dots,v_k\}$ but then we may use ${\leq 2}$-robustness and choose $\phi$ to colour $u_1$ the value $0$, which will force all remaining vertices in $\{u_1,\dots,u_k\}=\{v_1,\dots,v_k\}$ to be coloured $1$, which shows that $\{\phi(u_1),\dots,\phi(u_{i})\}=\{0,1\}\neq \{1\}=\{\phi(v_1),\dots,\phi(v_{j})\}$ as required.
So we assume that $\{u_1,\dots,u_k\}\neq \{v_1,\dots,v_k\}$, in which case $\{u_1,\dots,u_k\}$ and $\{v_1,\dots,v_k\}$  share at most one common vertex.
If they are disjoint, then at most one of $\{u_1,v_k\}$ and $\{v_1,u_k\}$ is a subyperedge: otherwise we get a $4$-cycle in $\mathbb{H}$.
There is no loss of generality in assume that $\{u_1,v_k\}$ is not a subhyperedge, in which case we can choose $\phi$ to colour both $u_1$ and $v_k$ the value~$0$.
Then $\{\phi(u_1),\dots,\phi(u_{i})\}=\{0,1\}\neq \{1\}=\{\phi(v_1),\dots,\phi(v_{j})\}$ as required.  Now consider the case where $\{u_1,\dots,u_k\}$ and $\{v_1,\dots,v_k\}$ share (exactly) one vertex $u$.
If $u$ is not $u_k$ then $\{u_k,v_1\}$ is not a subhyperedge (as otherwise we have a $3$-cycle in $\mathbb{H}$), and so we can choose a $(k-1)$-in-$k$ satisfaction $\phi$ giving both $u_k$ and $v_1$ the colour $0$, so that $\{\phi(u_1),\dots,\phi(u_{i})\}=\{1\}\neq \{0,1\}$, while $0\in\{\phi(v_1),\dots,\phi(v_{j})\}$.
If $u=u_k$ and $v_k\neq u_k$ then $\{u_1,v_k\}$ is not a hyperedge and we can choose a $(k-1)$-in-$k$ satisfaction $\phi$ giving both $u_1$ and $v_k$ the colour $0$ so that $0\in\{\phi(u_1),\dots,\phi(u_{i})\}$ while $\{\phi(v_1),\dots,\phi(v_{j})\}=\{1\}$.
If $u=u_k=v_k$, then the assumption that $i\geq j$ and $\{u_1,\dots,u_{i}\}$ and $\{v_1,\dots,v_{j}\}$ are not linked $(k-1)$-sets ensures $j<k-1$.  Then use the fact that $\{u_1,v_{k-1}\}$ is not a subhyperedge, to again find the required $(k-1)$-in-$k$ satisfaction.
\end{proof}

Recall that $S_c(a_1\dots a_k)=S_\mathbb{E}$, where $\mathbb{E}$ denotes the $k$-uniform hypergraph with a single hyperedge.  A $k$-uniform hyperforest $\mathbb{F}$ that is nontrivial contains at least one hyperedge, so that $\mathsf{V}(S_\mathbb{F})$ contains $S_\mathbb{E}$.  The next lemma shows the converse, from which it follows that all $k$-uniform hyperforest semirings generate the same variety.
\begin{lem}\label{lem:forest2}
If $k>2$ and $\mathbb{F}$ is a $k$-uniform hyperforest, then
$S_\mathbb{F}\in\mathsf{V}(S_c(a_1\dots a_k))$.
\end{lem}
\begin{proof}
It can be shown by induction \cite[Lemma 2.8]{hamjac:hyper} that $\mathbb{F}$ is an induced sub-hypergraph of a direct power $\mathbb{E}^T$ of $\mathbb{E}$, for some index set $T$ that is finite in the case that $\mathbb{F}$ is finite.  Thus each vertex $u$ of $\mathbb{F}$ may be associated with an $T$-tuple $\bar{u}$ of vertices of $\mathbb{E}$, in such a way that $\{u_1,\dots,u_k\}$ is a hyperedge of $\mathbb{F}$ if and only if for all $i\in T$ the tuple $\{\bar{u}_1(i),\dots,\bar{u}_k(i)\}$ is a hyperedge of $\mathbb{E}$.  (Of course, there is just one hyperedge in $\mathbb{E}$, so this means $\{\bar{u}_1(i),\dots,\bar{u}_k(i)\}=\{a_1,\dots,a_k\}$.)  In fact we may assume slightly more than this, as the representation $\mathbb{F}\leq \mathbb{E}^T$ can be shown to be ${\leq}2$-robust, in the sense that if $u$ and $v$ are distinct vertices that do not lie in the same hyperedge, then we may find $i,j\in T$ such that $\bar{u}(i)=\bar{v}(i)$ and $\bar{u}(j)\neq \bar{v}(j)$.  This is not directly shown in the proof of \cite[Lemma 2.8]{hamjac:hyper}, but it is easy to verify as we now sketch.  For every such $u,v$, we wish to find maps $\phi_1,\phi_2$ from $\mathbb{F}$ to $\mathbb{E}$ with $\phi_1(u)=\phi_1(v)$ and $\phi_2(u)=\phi_2(v)$; then we may include $\phi_1,\phi_2$ as coordinates in $T$, with $\bar{w}(\phi_i)=\phi_i(w)$ for all vertices $w$, just as in Lemma \ref{lem:HinS7}.
The proof of \cite[Lemma 2.8]{hamjac:hyper} is inductive, using the fact that every hyperforest contains a leaf.  Specifically we may find a series $\mathbb{E}\cong \mathbb{F}_0\leq \mathbb{F}_1\leq\dots\leq \mathbb{F}_n=\mathbb{F}$ of hyperforests, each obtained from its predecessor by adjoining a new leaf.  The ${\leq}2$-robustness condition holds at the base case, and if it holds in $\mathbb{F}_i$, then it is very straightforward to show it holds in $\mathbb{F}_{i+1}$ as there is almost complete freedom for how we extend our desired maps $\phi_1$ and $\phi_2$ to the new vertices added in the introduction of the new leaf $e$.  The only constraint is that our bijection from $e$ to $\{a_1,\dots,a_k\}$ has to agree on the (at most one) vertex common to $e$ and $\mathbb{F}_i$, if such a vertex exists.  The details are almost identical to that in \cite[Lemma 2.8]{hamjac:hyper} and we omit them.

The remaining arguments are very similar to those in the proof of Lemma \ref{lem:HinS7} so we again give only sketch  details.  We let $A$ be the subsemiring of $\mathbb{E}^T$ generated by the elements $\bar{u}$ for vertices $u$ in $\mathbb{F}$, and let $J$ denote the set of elements in $A$ with a zero coordinate, which is an additive and multiplicative ideal, so that we have a well defined quotient $A/J$.  We claim that $A/J\cong S_\mathbb{F}$.  Properties (1)--(4) for $S_\mathbb{F}$ are verified of $A/J$ immediately due to the fact that $A/J$ is commutative, and the ${\leq}2$-robustness property of the representation of $\mathbb{F}$ as an induced substructure of $\mathbb{E}^T$:  property (4), on linked subhyperedges, is because if $\{u_1,\dots,u_{k-1}\}$ and $\{v_1,\dots,v_{k-1}\}$ are linked via $w
$, and $\bar{w}(i)=a_\ell$, then $\bar{u}_1\dots\bar{u}_{k-1}(i)=\prod_{j\neq \ell}a_j=\bar{v}_1\dots\bar{v}_{k-1}(i)$.  Thus $A/J$ is a homomorphic image of $S_\mathbb{F}$.  The fact that it is isomorphic follows by an almost identical process to the final stages of the proof of Lemma \ref{lem:HinS7}, using the ${\leq}2$-robustness of the representation of $\mathbb{F}$ as an induced substructure of $\mathbb{E}^T$.  We omit the details.
\end{proof}
\begin{remark}\label{rem:monoidforest}
Lemma \ref{lem:forest2} holds also if the monoids $M_\mathbb{F}$  and $M_c(a_1\dots a_k)$ are used in place of $S_\mathbb{F}$ and $S_c(a_1\dots a_k)$, with trivial amendment to the proof.
\end{remark}

In the following results, a \emph{cyclic element} will mean an element $g$ satisfying $g^n=g$ for some
$n$, which in the periodic case is equivalent to lying a subgroup of the multiplicative reduct of a semiring.
\begin{defn}\label{def:1in3}
We say that the \emph{$1$-in-$3$ property} holds in a semiring $S$ at a point $c\in S$ if $c$ is noncyclic and there is an element $d$ such that  whenever $x,y,z\in S$ have all permutations of the product $xyz$ below $c$ in the $+$-semilattice order, then precisely one of $x,y,z$ equals $d$. The $2$-in-$3$ property is defined similarly but with precisely two of $x,y,z$ equal to $d$.   We say that the $1$-in-$3$ (respectively, the $2$-in-$3$) property holds in $S$ if it holds at every noncyclic element of $S$.
\end{defn}
The 1-in-3 property holds in $S_7$ for the unique noncyclic element $c:=a$ via the element $d:=a$, while the $2$-in-$3$ property holds using $d=1$.  The $2$-in-$3$ property holds in $S_c(abb)$ at the element $c:=abb$ via $d:=b$, while the $1$-in-$3$ property holds when $d:=a$.  The $1$-in-$3$ and $2$-in-$3$ properties hold vacuously at all other choices of $c$ in $S_c(abb)$ as all noncyclic elements admit no $x,y,z$ with $xyz=c$.  The $1$-in-$3$ property also holds in $B_2^1$: the noncyclic elements are $a$ and $b$, and up to symmetry we may consider $xyz,xzy,yxz,yzx,zxy,zyx\leq a$, from which it easily follows that $(x,y,z)\in\{(a,1,1),(1,a,1),(1,1,a)\}$.  Moreover the same holds more generally in $B_2^1(G)$ for any group $G$: as a semigroup this is just the Brandt semigroup over the group $G$, with adjoined identity.  As with $B_2$, the semigroup $B_2^1(G)$ becomes a flat semiring by way of the natural inverse semigroup order, and the added identity element in $B_2^1(G)$ will sit beneath the two nonzero idempotents of $B_2(G)$ as in Figure~\ref{fig:B21}.  All other elements are pairwise incomparable and are covered by $0$, as $a$ and $b$ are in Figure~\ref{fig:B21}.

\begin{lem}\label{lem:H23}
Let $\mathbb{H}$ be a $3$-uniform hypergraph of girth at least $5$, and let $S$ be an ai-semiring of finite $+$-height, finite period and satisfying the $1$-in-$3$ or $2$-in-$3$ property.  Assume further that the noncyclic elements form an order ideal in the $+$-order.
If $S_\mathbb{H}\in \mathsf{V}(S)$ then $\mathbb{H}$ is $2$-in-$3$-satisfiable.
\end{lem}
\begin{proof}
It suffices to assume that all vertices of $\mathbb{H}$ lie within a hyperedge, as we may trivially extend any $2$-in-$3$-satisfaction on the non-isolated vertices to any isolated vertices.

 Assume $S_\mathbb{H}\in \mathsf{V}(S)$. So there is a substructure $A\leq S^P$ for some finite set $P$ and $A$ is a homomorphic to $S_\mathbb{H}$, under some homomorphism $\psi:A\to S_\mathbb{H}$.  For every $x\in S_\mathbb{H}$, let $\bar{x}$ denote the largest element of $\psi^{-1}(x)$, which exists because $S$ has finite height (in the $+$-order).  As $\ba^2=0$ there is at least one $p\in P$ such that $\bar{\ba}(p)$ is noncyclic in $S$; call this element $c$.  We now invoke the 1-in-3 property, though note the argument for 2-in-3 as we go.  Using the 1-in-3 property, there is an element $d$ such that whenever $x,y,z$ have all permutations of the product $xyz$ beneath $c$, then precisely one of $x,y,z$ is $d$.  We colour the vertices of $\mathbb{H}$ as follows: if $\bar{\ba}_v(p)=d$ then we colour $v$ by $0$, and otherwise we colour $v$ by $1$.  (In the case of the $2$-in-$3$ property, the role of $0$ and $1$ is switched.)  We claim this is a $2$-in-$3$ colouring.  Now, all elements $\ba_u$ lie within a hyperedge $\{u,v,w\}$ and $\bar{\ba}_u\bar{\ba}_v\bar{\ba}_w(p)\leq \bar{\ba}(p)=c$.  So the $1$-in-$3$ property ensures that exactly one of $u,v,w$ is coloured $0$, as required.  Thus $\mathbb{H}$ is $2$-in-$3$-colourable, as required.
 \end{proof}
 Again the reader will readily verify that this result and proof also holds if $S_\mathbb{H}$ is replaced by $M_\mathbb{H}$ so that the result holds in the setting of semirings with identity also.
 Now we may prove one of our main nonfinite axiomatisability results.
\begin{thm}\label{thm:hard}
Let $S$ be a periodic ai-semiring of finite $+$-height with the following properties\up:
\begin{enumerate}
\item The noncyclic elements are an order ideal in the $+$-order\up;
\item  The $1$-in-$3$ property holds on $S$ or the $2$-in-$3$ property holds on $S$\up;
\item Either $S_7$ or $S_c(abb)$ is contained in $\mathsf{V}(S)$.
\end{enumerate}
Then it is \texttt{NP}-hard to distinguish the finite algebras in the subvariety $\mathsf{V}(S_c(abb))\wedge \mathsf{V}(S_7)$ from those that are not in $\mathsf{V}(S)$.
\end{thm}
\begin{proof}
Theorem \ref{thm:facts} shows that it is \texttt{NP}-hard to distinguish those $3$-uniform hypergraphs of girth at least $5$ that are ${\leq 2}$-robustly 2-in-3 satisfiable from those that are not 2-in-3 satisfiable.
Lemma \ref{lem:HinS7} shows that when a $3$-uniform hypergraph $\mathbb{H}$  of girth at least $5$ is ${\leq 2}$-robustly 2-in-3 satisfiable then $S_\mathbb{H}\in \mathsf{V}(S_c(abb))\wedge \mathsf{V}(S_7)$, while Lemma \ref{lem:H23} shows that when $\mathbb{H}$ is not 2-in-3 satisfiable then $S_\mathbb{H}\notin \mathsf{V}(S)$.
\end{proof}
As noted, the semirings $S_c(abb)$, $S_7$ and ${B}_2^1(G)$ (for any group $G$ of finite exponent) satisfy the $1$-in-$3$ property, and all have finite $+$-height.  The non-cyclic elements in each case form an order ideal, and either $S_7$ or $S_c(abb)$ are subsemirings.  Thus we have the following corollary to Theorem~\ref{thm:hard}.
\begin{cor}\label{cor:hard}
Each of the ai-semirings $S_c(abb)$, $S_7$ and ${B}_2^1(G)$ \up(for any  group $G$ of finite exponent\up) generates a variety with \texttt{NP}-hard variety membership.
\end{cor}
\begin{remark}
The computational problem of deciding membership in the variety generated by a finite flat algebra of finite signature and with absorbing $0$ is in \texttt{NP}.  Thus the membership problem for finite algebras in  $\mathsf{V}(S_c(abb))$ and in $\mathsf{V}(S_7)$ is \texttt{NP}-complete.
\end{remark}
\begin{proof}
 \texttt{NP}-hardness follows immediately from Theorem \ref{thm:hard}.  For membership in $\texttt{NP}$, let $A$ and $B$ be finite ai-semirings and with $B$ a flat semiring.  A certificate for membership of $B$ in $\mathsf{V}(A)$ is as follows: for each pair $a\neq b$ in $B$ we give a congruence $\theta$ such that $B/\theta$ is a flat semiring, $a\notin b/\theta$ and separating homomorphisms witnessing membership of the partial algebra $(B/\theta)\backslash\{0\}$ in $\mathsf{SP}(A\backslash \{0\})$; see Willard \cite[Theorem 1.2]{wil}.
\end{proof}
Corollary \ref{cor:hard} shows that $S_7$, $S_c(abb)$ and ${B}_2^1$ (and many other semirings) do not have a finite basis for their equations: if they did, then the finite basis would provide a polynomial time algorithm to determine membership in their variety, which would contradict Corollary \ref{cor:hard}.  Technically this assumes $\texttt{P}\neq \texttt{NP}$, though if we verify first order reductions throughout, we could observe that a finite basis would place the membership problem in $\texttt{AC}^0$ and use the known fact that $\texttt{AC}^0\subsetneq \texttt{NP}$.  We are, however, in a position to give a direct proof that these semirings, and those of the form $S_c(a_1\dots a_k)$ for $k>2$ are NFB and in some cases INFB relative to a broad class of ai-semirings, including all that satisfy the conditions of Corollary~\ref{cor:hard}, and in fact all finite flat semirings.  The \emph{index} of a semiring $S$, if it exists, will be the smallest $k$ such $S\models x^k\approx x^{k+p}$ for some $p$.  When $S$ is finite then its index exists and is obviously at most $|S|$.

\begin{thm}\label{thm:p3}
Let $S$ be any finite ai-semiring whose noncyclic elements form an order ideal $I$, and let $k'$ be the index of $S$.  If $S_c(a_1\dots a_{k})\in\mathsf{V}(S)$ for some $k\geq \max(k',3)$ then $S$ has no finite basis for its equational theory.
\end{thm}
\begin{proof}
Let $n$ be an arbitrary positive integer, and let $m$ denote the number $k\binom{kn}{2}$.  Let $\ell$ denote the number of elements in $S$.  By Theorem~\ref{thm:facts}(2), there exists a $k$-uniform hypergraph $\mathbb{H}_{n,\ell}=(V_{n,\ell},E_{n,\ell})$ of chromatic number greater than $\ell$ and girth greater than $m$.  Let $T$ be a subsemiring of $S_{\mathbb{H}_{n,\ell}}$ generated by at most $n$ elements; let $T_{\rm gen}$ denote the chosen set of generators.  For each $t\in T_{\rm gen}$ we may select $k$ generators $\ba_{u_{t,1}},\dots, \ba_{u_{t,k}}$ of $S_{\mathbb{H}_{n,\ell}}$ such that $t$ is equal to $\ba_{u_{t,1}}\dots\ba_{u_{t,i}}$ for some $i$ between $1$ and $k$ and $\{u_{t,1},\dots,u_{t,k}\}$ is a hyperedge.  In the case that $1<i<k-1$ the choice of $u_{t,i+1},\dots,u_{t,k}$ is unique because Lemma \ref{lem:hyperprop2}(I) implies that $\{u_{t,1},\dots,u_{t,i}\}$ extends to a unique hyperedge of $\mathbb{H}_{n,\ell}$ in this case.  In the case where $i=k-1$ the choice of $u_{t,k}$ is unique, but (by the defining property (4) for~$S_{\mathbb{H}_{n,\ell}}$) we may replace $\{u_{t,1},\dots,u_{t,k-1}\}$ by any other subhyperedge of size $k-1$ that is linked to $\{u_{t,1},\dots,u_{t,k-1}\}$.  In the case where $i=k$ then we may choose any hyperedge, while when $i=1$ then we may choose any hyperedge extending $\{u_{t,1}\}$.
In all cases however, we fix a choice and can generate $t$ from $\ba_{u_{t,1}},\dots, \ba_{u_{t,k}}$.  We consider now the subsemiring $T^+$ of $S_{\mathbb{H}_{n,\ell}}$ generated by $\{\ba_{u_{t,i}}\mid t\in T_{\rm gen},  i\leq k\}$, and let $\mathbb{G}$ denote the subhypergraph of $\mathbb{H}_{n,\ell}$ induced by the vertex set $V_\mathbb{G}:=\{u_{t,i}\mid t\in T_{\rm gen},  i\leq k\}$.
Now $|V_\mathbb{G}|\leq nk<m=k\binom{nk}{2}$ so that the condition on the girth of $\mathbb{H}_{n,\ell}$ ensures that~$\mathbb{G}$ is a hyperforest.
By Lemma \ref{lem:forest2}, we have that $S_\mathbb{G}$ lies in $\mathsf{V}(S_c(a_1\dots a_k))\subseteq \mathsf{V}(S)$.
Unfortunately, this information is not enough for us to conclude the proof.
The semiring $T^+$ is \emph{almost} equal to $S_\mathbb{G}$ except that, due to of the nature of the defining rules for the constructions, the semiring $T^+$ may carry the ``shadows'' of hyperedges in $\mathbb{H}_{n,\ell}$ that are not themselves in $\mathbb{G}$.  The issue arises in Rules~(1) and~(4) in the definition.
In the case of Rule~(1), there may be pairs of vertices $\{u,v\}$ in $\mathbb{G}$ that are subhyperedges in $\mathbb{H}_{n,\ell}$ but not in~$\mathbb{G}$, so that $\ba_u\ba_v=0$ in $S_\mathbb{G}$ but not in $T^+$.
In the case of Rule~(4)  there may be pairs of $(k-1)$-element sets of vertices $\{u_1,\dots,u_{k-1}\}$ and $\{v_{1},\dots,v_{k-1}\}$ in~$\mathbb{G}$ that are linked in $\mathbb{H}_{n,\ell}$ but not in~$\mathbb{G}$: this happens when the unique $k^{\rm th}$ vertex $w$ that completes these as hyperedges is not in the vertex set of $\mathbb{G}$.
In this situation  $T^+$ will have $\ba_{u_1}\dots\ba_{u_{k-1}}=\ba_{v_1}\dots\ba_{v_{k-1}}\neq 0$ while in $S_\mathbb{G}$ these two products will be~$0$.
To circumvent these issues we instead construct a larger subhyperforest $\mathbb{G}^+$ of $\mathbb{H}_{n,\ell}$ that contains all the corrections to these anomalies for~$\mathbb{G}$ in comparison to $\mathbb{H}_{n,\ell}$ so that $T^+$ embeds into $S_{\mathbb{G}^+}$.
We do not need to argue that  $S_{\mathbb{G}^+}$ is a subsemiring of $S_{\mathbb{H}_{n,\ell}}$---it will not typically be so---only that $S_{\mathbb{G}^+}$, and therefore $T^+$, lies in the variety of $S_c(a_1\dots a_k)$.
But this last fact follows immediately from Lemma \ref{lem:forest2} and the fact that $\mathbb{G}^+$ is a hyperforest.
So it remains to actually construct $\mathbb{G}^+$.   We define $\mathbb{G}^+$ as the induced subhypergraph of $\mathbb{H}_{n,\ell}$ on the following set $V_{\mathbb{G}^+}$ of vertices:
\[
\bigcup\{e\in E_{n,\ell}\mid 2\leq |e\cap V_\mathbb{G}|\}.
\]
Note that $V_\mathbb{G}\subseteq V_{\mathbb{G}^+}$ because we ensured that every $v\in V_\mathbb{G}$ is contained in some hyperedge $e$ of $\mathbb{G}$ and then $2\leq k=|e\cap V_\mathbb{G}|$ gives $v\in e\subseteq V_{\mathbb{G}^+}$.  We now claim that $|V_{\mathbb{G}^+}|\leq m$.  To see this, note that if $e,f\in E_{n,\ell}$ have $|e\cap V_\mathbb{G}|\geq 2$ and $|f\cap e|\geq 2$ then $e=f$ because no two distinct hyperedges share more than one vertex.  Thus we have
\[
V_{\mathbb{G}^+}=\bigcup\{e\in E_{n,\ell}\mid \text{$e$ extends a 2-element subset of $V_\mathbb{G}$}\},
\]
so that $|V_{\mathbb{G}^+}|\leq k\binom{|V_\mathbb{G}|}{2}\leq k\binom{kn}{2}=m$.  Again then, the condition on the girth of~$\mathbb{H}_{n,\ell}$ ensures that $\mathbb{G}^+$  is a hyperforest, so that $S_{\mathbb{G}^+}\in \mathsf{V}(S_c(a_1\dots a_k))$ by Lemma~\ref{lem:forest2}.  This time however we have that~$T^+$ is a subsemiring of $S_{\mathbb{G}^+}$ because the nonzero elements of $T^+$ are products of generators that correspond to subhyperedges of~$\mathbb{H}_{n,\ell}$ but lying within~$V_\mathbb{G}$, and every subset of $V_\mathbb{G}$ that is a subhyperedge in $\mathbb{H}_{n,\ell}$ is also a subhyperedge in~$\mathbb{G}^+$.  Of course, there may now be subsets of $V_{\mathbb{G}^+}$ that are subhyperedges in $\mathbb{H}_{n,\ell}$ but not in~$\mathbb{G}^+$, but we did not need to correct these: only those that corresponded to elements of $T^+$.

We have shown that $n$-generated subsemirings of $S_{\mathbb{H}_{n,\ell}}$ lie in $\mathsf{V}(S_c(a_1\dots a_k))\subseteq \mathsf{V}(S)$.  Now we show that $S_{\mathbb{H}_{n,\ell}}\notin \mathsf{V}(S)$; we attempt to follow the proof of Lemma~\ref{lem:H23}.  As in that proof, let $A\leq S^Q$ for some finite set $Q$ be such that $\phi:A\to S_{\mathbb{H}_{n,\ell}}$ is a surjective homomorphism, and let $\bar{a}$ be the maximum element of $\phi^{-1}(\ba)$.  As $\ba$ is not cyclic it follows that there is $q\in Q$ such that $\bar{a}(q)$ is not cyclic.

For each $u\in V_{n,\ell}$, fix an element $a_u\in \phi^{-1}(\ba_u)$.  Because $\mathbb{H}_{n,\ell}$ is not $\ell$-colourable and the map $u\mapsto a_u(q)$ maps into $S$ (of size $\ell$), there must be a hyperedge $\{u_1,\dots,u_k\}$ such that $a_{u_1}(q)=\dots=a_{u_k}(q)$.  Then $(a_{u_1}\dots  a_{u_k})(q)$ is cyclic, because  $A\models x^k\approx x^{k+p}$.  But $(a_{u_1}\dots  a_{u_k})(q)\leq \bar{a}(q)$ and $\bar{a}(q)$ is not cyclic, which then contradicts the assumption that the noncyclic elements form an order ideal of~$S$.  Thus $S_{\mathbb{H}_{n,\ell}}\notin \mathsf{V}(S)$ as required.
\end{proof}
\begin{remark}\label{rem:p3monoid}
Theorem \ref{thm:p3} also holds in the monoid signature, where $M_c(a_1\dots a_{k})$ replaces $S_c(a_1\dots a_{k})$.
\end{remark}
\begin{proof}
We have noted throughout that the relevant results and arguments to prove Theorem \ref{thm:p3} hold when $M_c(a_1\dots a_{k})$ replaces $S_c(a_1\dots a_{k})$ and the proof uses $M_{\mathbb{H}_{n,\ell}}$ in place of $S_{\mathbb{H}_{n,\ell}}$; see Remark \ref{rem:monoidforest} (in place of Lemma \ref{lem:forest}), and note that the final paragraph of the proof of Theorem \ref{thm:p3} holds because if $M_{\mathbb{H}_{n,\ell}}\in \mathsf{V}(S)$ in the signature $\{+,\cdot,1\}$, then $S_{\mathbb{H}_{n,\ell}}\in \mathsf{V}(S)$ in the signature $\{+,\cdot\}$, which we showed is not true.
\end{proof}
While $S_7$ is not INFB, Theorem \ref{thm:p3} shows it is relatively INFB for a broad class of semirings.
\begin{cor}\label{cor:rinfb}
The semiring $S_7$, is inherently nonfinitely based relatively to the property of being generated by a flat semiring.
\end{cor}
\begin{proof}
This follows immediately from Proposition \ref{pro:SinM} (with $S_7=M_c(a)$) and Theorem \ref{thm:p3}.
\end{proof}
Lemma \ref{lem:0directjoin} shows that we may also replace ``a flat semiring'' by ``a finite family of flat semirings'' in Corollary \ref{cor:rinfb}.  We may also visit the remaining cases from Corollary \ref{cor:hard}; the result follows already from Corollary \ref{cor:hard}, assuming $\texttt{P}\neq\texttt{NP}$, but follows unconditionally from Theorem \ref{thm:p3}.
\begin{cor}\label{cor:NFBtake2}
The ai-semirings $S_c(abb)$ and ${B}_2^1(G)$ \up(for any  group $G$ of finite exponent\up) are not finitely based.
\end{cor}
The next corollary provides a complete classification of the finite basis property in one natural case.
\begin{cor}\label{cor:monoid}
If a finite flat semiring has a multiplicative identity element, then it is finitely based if and only if it is the flat extension of a finite group whose nilpotent subgroups are abelian.
\end{cor}
\begin{proof}
This is an immediate consequence of Proposition \ref{pro:monoid}, Corollary \ref{cor:rinfb} and the fact that a flat extension of a finite group is finitely based if and only if all of its nilpotent subgroups are abelian.
\end{proof}
In Section \ref{sec:group} we will see that we can add the further equivalent condition of ``not being INFB relative to being generated by a finite flat semiring'' as well.
\section{The semiring $S_7$}\label{sec:S7}
In \cite{ZRCSD} it is shown that all ai-semirings of order at most $3$, with the possible exception of $S_7$, are finitely based.
Combining this with Corollary \ref{cor:rinfb} we can state the following result.
\begin{cor}\label{cor:3element}
A finite ai-semiring of order at most $3$ is finitely based if and only if it is not $S_7$.
\end{cor}
Three-elements is also the smallest possible size of a generator for non-polynomial time variety membership given that all $2$-element algebras (of finite signature) have a finite basis for their equations \cite{lyn}; the only other known example was given in \cite[\S8]{jac:SAT}.
In view of the special status that $S_7$ appears to have in terms of small ai-semirings, it seems pertinent to provide some further investigation to the equational properties of $S_7$.

First we observe that it is easily seen that $S_7$ as co-$\texttt{NP}$-complete equation checking problem; it is not the smallest with this property, as even the two element distributive lattice $D_2$ has this property \cite{BHR}.

Let $\mathscr{S}$ be a family of nonempty subsets of a set $\{x_1,\dots,x_n\}$.  For each $s\in \mathscr{S}$, let $\bw_s$ denote the product of the variables in $s$ (so, if $s=\{x,y,z\}$, then $\bw_s=xyz$, noting that as $S_7$ is commutative, the order of appearance in the product is not important).  Let $t_\mathscr{S}$ denote the term
\(
\sum_{s\in \mathscr{S}}\bw_s.
\)
\begin{lem}
Let $\mathscr{S}$ be a family of nonempty subsets of a set $\{x_1,\dots,x_n\}$, with each $x_i$ appearing in at least one element of $\mathscr{S}$.  Then $S_7\models t_\mathscr{S}\approx t_\mathscr{S}^2$ if and only if there is a subset $Y\subseteq \{x_1,\dots,x_n\}$ that intersects each $s\in\mathscr{S}$ exactly once.
\end{lem}
\begin{proof}
For an assignment $\nu:\{x_1,\dots,x_n\}\to S_7$, let $Y$ denote those variables assigned $a$.  Then $\nu(t_\mathscr{S})$ is nonidempotent
if and only if $\nu(t_\mathscr{S})=a$, if and only if $\nu(x_i)=a$ for all $i$, if and only if $Y$ intersects each $s\in\mathscr{S}$ exactly once.
\end{proof}
When each $s\in\mathscr{S}$ has size exactly $3$, this is the well known \texttt{NP}-complete problem 1-in-3SAT.  Thus we have the following consequence.
\begin{cor}
Equation checking is co-\texttt{NP}-complete for $S_7$.
\end{cor}

This idea generalises to a full combinatorial characterisation of the equational theory of $S_7$.  By distributivity, all ai-semiring terms in variables $X=\{x_1,x_2,\dots\}$ are finite sums of words in $X^+$, thus for the purposes of characterisation we may consider an \emph{ai-semiring identity} (${\bf AI}$-identity for short) over $X$ as an
expression of the form $\bu\approx \bv$, where $\bu, \bv\in \wp_{\rm fin}f(X^+)$, the nonempty finite subsets of $X^+$.
Thus we consider
$\bu_1+\cdots+\bu_k\approx \bv_1+\cdots+\bv_\ell$ synonymously with the ${\bf AI}$-identity
$\{\bu_i \mid 1\leq i\leq k\}\approx \{\bv_j \mid 1\leq j\leq \ell\}$.

Let $\bw\in X^+$ and $x\in X$. Then let
\begin{itemize}
\item $c(\bw)$ (the \emph{content} of $\omega$) denote the set of variables that occur in $\bw$.
\item $\occ(x, \bw)$ denotes the number of occurrences of $x$ in $\bw$.
\end{itemize}
These notations extend to sums of words in an obvious way though we need only
\[
c\Big(\sum_{1\leq i\leq n}\bw_i\Big)=\bigcup_{1\leq i\leq n}c(\bw_i).
\]
Now consider an ${\bf AI}$-term $\bw:= \bw_1 + \dots + \bw_m$ , where each $\bw_i\in X^+$.  We let $\delta(\bw)$ denote the set of nonempty subsets $Z$ of $c(\bw)$ such that both
\begin{itemize}
\item
$Z \cap c(\bw_i)$ is a singleton for every $\bw_i$  and
\item $\occ(x, \bw_i ) =1$ if $\{x\}=Z \cap c(\bw_i)$.
\end{itemize}

Let $M_2$ denote $2$-element flat semiring on $\{1,0\}$ with $0$ a multiplicative and additive zero and $1\cdot 1=1$; in the $M(W)$ notation, we have $M_2=M(1)$ where $1$ is the empty word. The solution of the equational problem for $M_2$
can be found in~\cite{sharen}.
\begin{lem}\label{l1} \up(\cite[Lemma~1.1~(iii)]{sharen}.\up)
Let $\bu\approx \bv$ be an ${\bf AI}$-identity. Then
$\bu\approx \bv$ holds in $M_2$ if and only if $c(\bu)=c(\bv)$.
\end{lem}

\begin{pro}\label{npro1}
Let $\bu\approx \bv$ be an ${\bf AI}$-identity. Then
$\bu\approx \bv$ holds in $S_7$ if and only if $c(\bu)=c(\bv)$ and $\delta(\bu)=\delta(\bv)$.
\end{pro}
\begin{proof}
Suppose that $\bu\approx \bv$ holds in $S_7$. Since $\{0, 1\}$ forms a subsemiring
of $S_7$ and is isomorphic to $M_2$, it follows from Lemma \ref{l1} that $c(\bu)=c(\bv)$.
Let $Z$ be an arbitrary element of $\delta(\bu)$.
Then $Z\subseteq c(\bu)$. It follows immediately that $Z\subseteq c(\bv)$.
Now consider the substitution $\varphi_Z: X\rightarrow S_7$: $\varphi_Z(x)=a$
if $x\in Z$ and $\varphi_Z(x)=1$ otherwise. Then $\varphi_Z(\bu)=a$.
Since $\bu\approx \bv$ holds in $S_7$, it follows that
$\varphi_Z(\bv)=\varphi_Z(\bu)=a$. This implies that $\varphi_Z(\bv_j)=a$
for every $\bv_j$ in $v$. Furthermore, for each $\bv_j$ in $\bv$,
there exists $x_j$ in $X$ such that $Z \cap c(\bv_j)=\{x_j\}$
and $\occ(x_j, \bv_j) =1$.
Thus $Z\in \delta(\bv)$ and so $\delta(\bu)\subseteq\delta(\bv)$.
Similarly, $\delta(\bv)\subseteq\delta(\bu)$.
We now conclude that $\delta(\bu)=\delta(\bv)$.

Conversely, assume that $c(\bu)=c(\bv)$ and $\delta(\bu)=\delta(\bv)$.
Let $\psi: X\rightarrow S_7$ be an arbitrary substitution.
Consider the following two cases:
\begin{itemize}
\item $\{x\mid x\in c(\bu), \psi(x)=0\}\neq \varnothing$.
Since $c(\bu)=c(\bv)$, it follows that $\psi(\bu)=\psi(\bv)=0$.

\item $\{x\mid x\in c(\bu), \psi(x)=0\}=\varnothing$.
Then $\psi(x)=a$  or $1$ for every $x$ in $c(\bu)\cup c(\bv)$.
Let $Z$ denote the set $\{x\mid x\in c(\bu), \psi(x)=a\}$.
If $Z=\varnothing$, then $\psi(\bu)=\psi(\bv)=1$. Otherwise,
we need to consider the following two subcases:
\begin{itemize}
\item[$\diamond$] $Z\in \delta(\bu)$. Then $\psi(\bu)=a$.
Since $\delta(\bu)=\delta(\bv)$, it follows that $Z\in \delta(\bv)$.
This implies that $\psi(\bv)=a=\psi(\bu)$.

\item[$\diamond$] $Z\notin \delta(\bu)$. Then $\psi(\bu)=0$.
Since $\delta(\bu)=\delta(\bv)$, it follows that $Z\notin \delta(\bv)$.
This implies that $\psi(\bv)=0=\psi(\bu)$.
\end{itemize}
\end{itemize}
We therefore have that $\bu\approx \bv$ holds in $S_7$.
\end{proof}

\section{Non-abelian $p$-groups imply nonfinitely based}\label{sec:group}
It follows from results in \cite[Theorem 7.3]{jac:flat} (stated slightly more directly in \cite[Remark~4.9]{jac:eqncomp}) that the flat extension of a finite group with a nonabelian Sylow subgroup is INFB relative to the property of being a variety generated by a finite flat semiring.  Unlike $S_7$, these examples are not commutative, but they do have a multiplicative reduct that is FB, due primarily to the Oates-Powell Theorem \cite{oatpow}.  We now observe that the same ideas yield a significantly stronger nonfinite basis result.
\begin{thm}\label{thm:pgroup}
Let $S$ be a finite ai-semiring.  If the multiplicative reduct of $S$ contains a nonabelian nilpotent subgroup, then every finite ai-semiring whose variety contains $S$ is nonfinitely based.
\end{thm}

This theorem will be proved over the remainder of this section.  We first establish some basic lemmas; the first is well known.
\begin{lem}\label{lem:order}
If $G$ is a periodic group and $\leq$ an order that is preserved by the multiplication of $G$, then $\leq$ is an antichain.  In particular, a subgroup of the multiplicative reduct of a finite semiring is always an antichain relative to the $+$ order.
\end{lem}
\begin{proof}
Assume $g\leq h$.  Then $gh^{-1}\leq e$.  Let $k$ denote $gh^{-1}$; we have $k\leq e$ and $k^n=e$ for some $n$.  Then $kk\leq ke=k$.  Then $k^3\leq k^2\leq k\leq e$ and so on: $k^i\leq k^j$ whenever $i>j$.  But then $e=k^n\leq k^{n-1}\leq\dots\leq k^2\leq k\leq e$, so that all of these inequalities are equalities.  Thus $gh^{-1}=e$, or equivalently, $g=h$.
\end{proof}
In the following $\mathsf{HS}(S)$ will denote, as usual, those algebras obtained as a homomorphic images of subalgebras of $S$.
\begin{lem}\label{lem:flatin}
Let $G$ be a finite nontrivial subgroup of the multiplicative reduct of an ai-semiring $S$.  Then $\flat(G)\in\mathsf{HS}(S)$.
\end{lem}
\begin{proof}
Let $S_G$ be the subsemiring generated by $G$, which consists of the elements of $G$ along with finite sums of elements of $G$ (this is clearly generated by $G$ and also a subsemiring).  Let $G^+$ denote all elements that can be written as a proper sum; that is, a sum involving at least two distinct elements of $G$; equivalently, $G^+=S_G\backslash G$, which is nonempty as $G$ is nontrivial.    We claim that $G^+$ is a multiplicative ideal of $\langle S_G;\cdot\rangle$ and also a filter of $\langle S_G;+\rangle$.  For the multiplicative reduct claim, note that if $s=g_1+\dots +g_n$ is a sum of distinct elements of $G$, where $n>1$, and $t$ is also a sum of distinct elements $h_1+\dots +h_m$, then $st=\sum_{i\leq n, j\leq m}g_ih_j$ and $g_1h_1\neq g_2h_1$, so that $st\in G^+$; by symmetry,  $ts\in G^+$ too.  For the filter property, simply note that if $s\leq g$ for some $g\in G$ and $s=g_1+\dots +g_n$ is as before, then $g_1\leq g$ and $g_2\leq g$, which contradicts Lemma \ref{lem:order}.  Thus $G^+$ is a multiplicative ideal and additive filter, so that we may take a Rees quotient (identifying all elements of $G^+$) and obtain an isomorphic copy of $\flat (G)$.
\end{proof}

The following lemma seems quite counterintuitive, as it shows that when a group arises within a quotient of (the multiplicative reduct of) a finite ai-semiring $S$, then it was already present as a subgroup of $S$.
\begin{lem}[The group quotient embedding lemma]\label{lem:reverse}
Let $S$ and $A$ be finite ai-semirings and let $G$ be a finite subgroup of the multiplicative reduct of~$S$.  If there is a surjective semiring homomorphism $\phi:A\to S$, then there is an injective semigroup homomorphism from $G$ into $A$.
\end{lem}
\begin{proof}
First recall that  the semigroup-theoretic $\mathscr{J}$-order is defined by $x\leq_\mathscr{J} y$ if there exists $u,v$ (possibly empty) such that $x=uyv$; see a test such as Howie \cite{how} for example.
Observe that $\phi^{-1}(G)$ is a subsemigroup $A_G$ of the multiplicative reduct of $A$.  Because $A$ is finite, $A_G$ contains a minimal idempotent $e_A$ with respect to the $\mathscr{J}$-order; evidently we must have $\phi(e_A)=e$, as $e_Ae_A=e_A$.  Moreover, $e_AA_Ge_A$ is a finite group $H$ with $\phi(H)=G$.  The proof will then be complete if we can show that $\phi$ is injective on $H$.  Let $N\leq H$ denote $\phi^{-1}(e)\cap H$; the kernel of $\phi$ restricted to $H$.  Consider any $g_1,g_2\in N$, and observe that $g_1+g_2$ is also in $N$ because $\phi(g_1+g_2)=\phi(g_1)+\phi(g_2)=e+e=e$ and $e_A(g_1+g_2)e_A=g_1+g_2$ by distributivity.  But then $g_1\leq g_1+g_2$ and $g_2\leq g_1+g_2$ implies $g_1=g_2=g_1+g_2$ by Lemma \ref{lem:order}, which shows that $|N|=1$ as required.
\end{proof}

\begin{lem}\label{lem:Gin}
Let $S$ and $T$ be finite ai-semirings with $S\in\mathsf{V}(T)$ and let $G$ be a subgroup of the multiplicative reduct of $S$.  Then $G$ lies in the quasivariety of the multiplicative reduct of $T$.
\end{lem}
\begin{proof}
There is $k\in \mathbb{N}$ and $A\leq T^k$ such that $\phi:A\onto S$.  By Lemma \ref{lem:reverse} we have that $G$ embeds into $A$ as a multiplicative subgroup.  Then $G$ is a subgroup of a power of the multiplicative reduct of $T$, or in other words is in the quasivariety generated by $T$ (as a multiplicative semigroup), as claimed.
\end{proof}

Now we may prove Theorem \ref{thm:pgroup}.
\begin{proof}[Proof of Theorem \ref{thm:pgroup}]
Let $S$ be a finite ai-semiring containing a non-abelian but nilpotent subgroup $G$.  Let $T$ be a finite ai-semiring with $S\in\mathsf{V}(T)$.  Ol$'$shanski\u{\i}~\cite{ols} proved that for each $n$ there is a finite group $G_n$ that does not lie in the quasivariety of $T$, but whose $n$-generated subgroups lie in the quasivariety of $G$.  (Technically, this is proved by Ol$'$shanski\u{\i} when $T$ is a group, but the argument uses only the cardinality of $T$, so holds in the broader class of all semigroups.)  Then the $n$-generated subsemirings of $\flat(G_n)$ lie in the variety of $\flat(G)$, which lies in $\mathsf{V}(S)\leq \mathsf{V}(T)$ by Lemma~\ref{lem:flatin}.  We wish to show that $\flat(G_n)\notin \mathsf{V}(T)$.  But this follows immediately from Lemma~\ref{lem:Gin}, because $G_n$ is not in the quasivariety of~$T$.
\end{proof}

Note that we also showed the following.
\begin{lem}\label{lem:localflat}
Let $G$ be a finite group and $S$ a finite semiring. The following are equivalent.
\begin{enumerate}
\item $\flat(G)\in \mathsf{V}(S)$,
\item $\flat(G)$ lies in the variety generated by the flat semirings in $\mathsf{HS}(S)$,
\item $G$ lies in the quasivariety of the subgroups of $S$.
\end{enumerate}
\end{lem}
\begin{proof}
Using Lemma \ref{lem:Gin} we have that (1) implies that $G$ is in the quasivariety of the multiplicative reduct of $S$, which implies that $G$ is in the quasivariety generated by the class of subgroups of the multiplicative reduct of~$S$; that is (3).  In turn (3) implies that $\flat(G)$ is in the variety generated by the flat extensions of subgroups of~$S$ by \cite{jac:flat}; but these are amongst the flat semirings in $\mathsf{HS}(S)$ by Lemma \ref{lem:flatin}; so (2) holds.  That (2) implies (1) is trivial.
\end{proof}

As noted in \cite[Remark~4.9]{jac:eqncomp}, the smallest nonabelian nilpotent groups have 8 elements, so the smallest nonfinitely based ai-semirings from this group-theoretic approach have $9$ elements.    It also follows immediately from \cite{jac:flat} and \cite{ols} that there are infinitely many examples of \emph{limit varieties} of semirings that are generated by flat extensions of groups: nonfinitely based varieties whose subvarieties are all finitely based.  Subsequent work in preparation will give greater detail of these and other kinds of limit variety of ai-semiring, though we observe that the properties of the 27-element group $G_3$ described in \cite[\S3.4]{jac:continuum} show that it yields a concrete example after applying the flat extension construction.  (The group $G_3$ is just the $2$-generated free Burnside group $B(2,3)$ of exponent~$3$, as described in Burnside \cite{bur} or Hall \cite{hal}.)   

The power semiring $\wp(G)$ of a finite group contains $G$ as a subgroup, so Theorem~\ref{thm:pgroup} will apply to show that $\wp(G)$ is not finitely based whenever $G$ contains a nonabelian Sylow subgroup.
\begin{example}\label{eg:q}
The power semiring of the quaternion group is nonfinitely based as a semiring, and is inherently nonfinitely based relative to the property of being generated by a finite ai-semiring, but not relative to generating a locally finite variety.
\end{example}
\begin{proof}
Theorem \ref{thm:pgroup} shows that $\wp(Q)$ is inherently nonfinitely based relative to the property of being generated by a finite semiring.
The idempotents in $\wp(G)$ are precisely the subgroups (or $\varnothing$) of $G$.  When all subgroups of $G$ are normal, such as for $G=Q$, the idempotents of $\wp(Q)$ are central.  This is not possible for a finite INFB semigroup: it is a trivial consequence of the first author's analysis \cite{jac:INFB} (based on Mark Sapir's celebrated classification of INFB finite semigroups \cite{sap1,sap2}), as no minimal finite INFB divisor has central idempotents.  Indeed, it follows from the main classification there, and the fact that $\wp(G)$ is a block group (see \cite{pin} for example), that~$\wp(G)$ is INFB as a semigroup if and only if $B_2^1$ is a divisor.  Thus $\wp(Q)$ is not INFB as a semigroup, hence not as a semiring either by Theorem~\ref{thm:INFB}(3).
\end{proof}
The arguments for Example \ref{eg:q} apply equally to any finite nonabelian group whose subgroups are all normal (that is, a \emph{Hamiltonian} group).
To conclude this section we note an extension of Corollary \ref{cor:monoid}, which gives a rare example of a variety for which the finite basis problem is moderately complex but for which we can give an effective characterisation.
\begin{thm}\label{thm:classification}
In the signature $\{+,\cdot,1\}$ of semirings with identity, the variety generated by flat semirings with identity is defined by the ai-semiring laws, along with the flat semiring laws \up(Lemma \ref{lem:flatvariety}\up) and $1x\approx x1\approx x$.  The following are equivalent for a finite semiring with identity $S$ in this variety\up:
\begin{enumerate}
\item $S$ is FB\up;
\item $S$ is completely regular and all nilpotent subgroups are abelian\up;
\item all nilpotent subgroups of $S$ are abelian and the variety of $S$ avoids $S_7$.
\end{enumerate}
\end{thm}
\begin{proof}
The variety generated by $S$ is equal to the variety generated by the set $\mathscr{S}$ of subdirectly irreducible quotients of $S$, which by Lemma \ref{lem:flatvariety} is a finite set of flat monoids.  But by Proposition \ref{pro:monoid}, a flat  semiring with identity either contains $S_7$ as a subsemiring (even when $1$ is in the signature) or is a flat extension of a group.  If $S_7$ is a subsemiring of a member of $\mathscr{S}$, then we may choose any $k$ greater than the index of $S$, and use the monoid version of Theorem \ref{thm:p3} (Remark \ref{rem:p3monoid}), using Proposition \ref{pro:SinM} to show that $S$ is not finitely based.  If $S$ (or equivalently, a member of $\mathscr{S}$) contains a nonabelian nilpotent subgroup, then $S$ is NFB by Theorem \ref{thm:pgroup}.  If $\mathscr{S}$ consists of flat extensions of groups whose nilpotent subgroups are abelian, then the variety of $S$ can be generated by the flat extension of the finite group $G$ obtained as the direct product of all subgroups of $S$.  All nilpotent subgroups of this semiring are abelian, so that $S$ is FB by \cite[Theorem 7.3]{jac:flat}.
\end{proof}

\section{Open problems and future directions}\label{sec:problems}
In this article we have shown that finite ai-semirings can have no finite basis for their equations for reasons that are nothing to do with the corresponding property for their semigroup reducts.  While a great many ai-semirings are covered by the results, there is a sense in which this is a start rather than a conclusion.  The finite basis property is clearly very rich for ai-semirings and the present results arguably open up more new directions than they close off.  An obvious overall goal is a complete classification of the finite basis property for ai-semirings, though this is probably too ambitious to provide good direction for future progress.
While the finite basis problem for general algebras is undecidable (McKenzie~\cite{mck}), we offer no speculations within the class of all ai-semirings despite the promising positive steps in Theorem \ref{thm:classification}.  Instead we list a number of more achievable problems that we feel might shape the area.

We have shown that $M_c(W)$ and $M(W)$ is always nonfinitely based (for finite, nonempty sets of words $W$), and that $S_c(W)$ is very often nonfinitely based.
\begin{problem}\label{prob:1}
\begin{enumerate}
\item For which finite sets of words $W$ is $S(W)$ finitely based as an ai-semiring?
\item When $W$ consists of powers of letters and words of length at most $2$, when is $S_c(W)$ finitely based?
\item More generally than 1 and 2\up: when a finite nilpotent semigroup carries a flat semiring structure, under what conditions is it finitely based?
\item More generally still, which finite flat semirings are finitely based?
\end{enumerate}
\end{problem}
It is also of interest to investigate Problem \ref{prob:1} in the context of the computational complexity of the membership problem.
\begin{problem}\label{prob:2}
When is the power semiring of a finite semigroup finitely based as an ai-semiring?
\end{problem}
The examples in the present article show that, unlike for semigroups, commutativity is not a barrier for the nonfinite basis property in the world of finite ai-semirings.  To some extent it even acts as a trigger for our arguments, as the hypergraph encoding makes inherent use of the commutativity of products formed from vertices.  On the other hand, the situation for flat groups \cite{jac:flat} depended explicitly on non-commutativity (of Sylow subgroups).  The following problem may be achievable.
\begin{problem}
Which finite multiplicatively-commutative ai-semirings are finitely based?  The restriction of this problem to nilpotent semirings, to flat semirings and to power semirings are natural restrictions of interest.
\end{problem}

Most of our examples do not have an additive zero element, while the important earlier investigations of Igor Dolinka \cite{dol0,dol1,dol2,dol3} included $0$ in the signature and axiomatically required that it be an additive identity and multiplicative zero.  This is of course not compatible with any flat ai-semiring structure though one can always extend a flat semiring to an ai-semiring by adjoining an additive identity (so that the $+$-semilattice reduct has height $2$).  It is not entirely clear how our methods apply in this case, and Problems \ref{prob:1}, \ref{prob:2} remain of interest (adding the required $0$ in the case of Problem \ref{prob:1}, and letting $\varnothing$ be the~$0$ in Problem \ref{prob:2}).  The finite basis property is even unclear in the case of adjoining a $0$ to $S_7$.
\begin{problem}\label{prob:3}
\begin{enumerate}
\item Is $S_7^0$ finitely based or not finitely based?
\item In the signature $\{+,\cdot\}$, what is the cardinality of the interval  $[\mathsf{V}(S_7),\mathsf{V}(S_7^0)]$ in the lattice of semiring varieties?
\item When is the finite basis (nonfinite basis) property stable under adjoining an additive $0$?
\end{enumerate}
\end{problem}
In this context, we now show that Theorem \ref{thm:pgroup} continues to hold in the setting of  semirings with $0$.
We first reprove a version of Lemma \ref{lem:flatin} in the semiring with~$0$ setting.
\begin{lem}\label{lem:flatin0}
Let $G$ be a finite nontrivial subgroup of the multiplicative reduct of an ai-semiring $S$ with $0$.  Then $\flat(G)^0\in\mathsf{HS}(S)$.
\end{lem}
\begin{proof}
The proof is very similar to that of Lemma \ref{lem:flatin}, so we condense steps where possible.
First note that $0\notin G$, as $0$ is a multiplicative zero for $S$ and $G$ is nontrivial.
Let $S_G$ be the subsemiring with $0$ generated by $G\cup\{0\}$, and let $G^+$ denote all elements that can be written as a proper sum of elements of $G$ and note that again $0\notin G^+$, as $0\leq s$ for all $s\in S_G$.  As before, $G^+$ is a multiplicative ideal of $\langle S_G\backslash \{0\};\cdot\rangle$ and also an additive filter of $\langle S_G;+\rangle$.  Then the equivalence relation $\theta:=\{(x,x)\mid x\in G\cup \{0\}\}\cup (G^+\times G^+)$ is a congruence and $S_G/\theta\cong \flat(G)^0$.
\end{proof}
We can now prove Theorem \ref{thm:pgroup} in the signature $\{+,\cdot,0\}$.
\begin{thm}\label{thm:pgroup0}
Theorem \ref{thm:pgroup} holds in the signature $\{+,\cdot,0\}$ of semirings with $0$.
\end{thm}
\begin{proof}
Lemma \ref{lem:order} holds without change.  In place of Lemma \ref{lem:flatin} we may use Lemma~\ref{lem:flatin0}.  The Group Quotient Embedding Lemma~\ref{lem:reverse} holds immediately, because every homomorphism in the signature $\{+,\cdot,0\}$ is a homomorphism in the reduct signature $\{+,\cdot\}$ and the conclusion of the lemma concerns only the semigroup reduct.  Similarly, Lemma \ref{lem:Gin} holds in the semiring with $0$ setting because if $S\in\mathsf{V}(T)$ there, then so also does $S\in\mathsf{V}(T)$ in the reduct signature $\{+,\cdot\}$, and  the conclusion again concerns only the $\cdot$-reduct.
Thus we have all of the ingredients required for the proof of Theorem \ref{thm:pgroup}, now using $\flat(G)^0$ in place of $\flat(G)$.
\end{proof}

Every inverse semigroup carries a natural order given by $x\leq y\Leftrightarrow xx^{-1}y=x$ and a number of authors have investigated when this order has compatible greatest lower bounds $\wedge$; see Leech \cite{lee}, but also Garvac$'$ki\u{\i} \cite{gar}.  This is also the basis of the recent approach by Volkov \cite{vol21}, and the starting point for the first author's investigations in \cite[\S7.8]{jac:flat} via \cite[\S7]{jacsto:PM}.  The standard Wagner-Preston representation for inverse semigroups extends to a representation of $\wedge$ as intersection, which gives rise to an ai-semiring structure where the role of addition is played by $\wedge$.  Garvac$'$ki\u{\i}'s results show that even without inverse,  the law $xv+uv+uy\leq xy$ (within ai-semirings) characterises semigroups of injective partial maps with intersection, where $+$ is intersection of relations.  The order on $B_2^1$ given in Figure \ref{fig:B21} is an example of these approaches.  Subsemirings include $B_2$ (dropping the element $1$), as well as the non-inverse subsemirings $B_0$ (dropping $b$ from $B_2$) and $P$ (dropping $ba$ from $B_0$).  The semigroups $B_2$, $B_0$, and $P$ have each played an important role in semigroup variety considerations: see  \cite{tra}, \cite{edm,leeB0} and \cite{golsap} (respectively) for example.  The semiring $P$ appears as $S_4$ in \cite{ZRCSD}, where a finite basis is given for its equations.
\begin{problem}\label{prob:4}
\begin{enumerate}
\item Resolve the finite or nonfinite basability of $B_0$ and $B_2$ as ai-semirings.
\item Is $B_2^1$ inherently nonfinitely based as an ai-semiring?
\item Which finite naturally semilattice-ordered inverse semigroups are finitely based, in either of the signatures $\{+,\cdot\}$ or $\{+,\cdot,0\}$?
\end{enumerate}
\end{problem}
Parts (1) and (2) of Problem \ref{prob:4} will be significant steps toward (3), especially given that the finitely based naturally semilattice-ordered Clifford semigroups are completely classified in \cite{jac:flat}. Problem \ref{prob:4}(1) is also a particular case of  Problem~\ref{prob:1}(4), as $B_0$ and $B_2$ are flat semirings.

\bibliographystyle{amsplain}

\begin{thebibliography}{99}
\bibitem{AGK} S. Abramsky, G. Gottlob, and P.G. Kolaitis, Robust constraint satisfaction and local hidden variables in quantum mechanics, in  IJCAI '13, pp. 440--446, 2013.
\bibitem{AEI} L. Aceto, Z. \'Esik and A. Ing\'olfsd\'ottir, The max-plus algebra of the natural numbers has no finite equational basis, Theoret. Comput. Sci. 293 (2003), 169--188.
\bibitem{andmik} H. Andr\'eka and Sz. Mikul\'as, Axiomatizability of positive algebras of binary relations, Algebra Universalis 66 (2011), 7--34.
\bibitem{BHR} P.A. Bloniarz, H.B. Hunt III, and D.J. Rosenkrantz, Algebraic structures with hard equivalence and minimization problems, J. Assoc. Comput. Mach. 31 (1984), 879--904.
\bibitem{bragur} J. Brakensiek and V. Guruswami. Promise Constraint Satisfaction: Structure Theory and a Symmetric Boolean Dichotomy. In Proceedings of the Twenty-Ninth Annual ACM-SIAM Symposium on Discrete Algorithms (SODA’18). Society for Industrial and Applied Mathematics, Philadelphia, PA, USA, 1782--1801.
 \bibitem{BKO} J. Bul\'{\i}n, A. Krokhin and J. Opr\v{s}al, Algebraic approach to promise constraint satisfaction. In Proceedings of the 51st Annual ACM SIGACT Symposium on the Theory of Computing (STOC '19), New York, NY, USA, 2019.
 \bibitem{bur} W. Burnside, ``Theory of Groups of Finite Order'', Cambridge Univ. Press, 2nd ed., 1911.
 \bibitem{con} J.H. Conway, ``Regular Algebra and Finite Machines'', Chapman \& Hall/CRC Mathematics. Chapman and Hall, Ltd., London, 1971.
 \bibitem{dol0} I. Dolinka, A nonfinitely based finite semiring, Internat. J. Algebra Comput. 17 (2007), 1537--1551.
 \bibitem{dol1} I. Dolinka, A class of inherently nonfinitely based semirings, Algebra Universalis 60 (2009), 19--35.
 \bibitem{dol2} I. Dolinka, A remark on nonfinitely based semirings, Semigroup Forum 78 (2009), 368--373.
 \bibitem{dol3} I. Dolinka, The finite basis problem for endomorphism semirings of finite semilattices with zero, Algebra Universalis 61 (2009) 441--448.
 \bibitem{edm} C.C. Edmunds,  Varieties generated by semigroups of order four, Semigroup Forum 21 (1980), 67--81.
\bibitem{erdhaj} P. Erd\H{o}s, A. Hajnal, On chromatic number of graphs and set-systems. Acta Mathematica Academiae Scientiarum Hungaricae Tomus 17, 61--99 (1966).
\bibitem{gar} V.S. Garvac$'$ki\u{\i}, $\cap$-semigroups of transformations, Theory of Semigroups and its
Applications, No. 2, Izdat. Saratov. Uni., Saratov (1971), pp. 2--13 (in Russian).
\bibitem{GPZ} S. Ghosh, F. Pastijn and X.Z. Zhao,
Varieties Generated by Ordered Bands I, Order 22 (2005), 109--128.
\bibitem{golsap} \`{E}.A. Golubov and M.V. Sapir, Varieties of finitely approximable semigroups. Soviet Math. (Iz. VUZ)
26(11) (1982), 25--36.
\bibitem{hal} M. Hall Jr., Theory of Groups, Chelsea Publishing Company, New York, 2nd ed., 1976.
\bibitem{ham} L. Ham, Gap theorems for robust satisfiability: Boolean CSPs and beyond, Theoretical Computer Science 676 (2017), 69--91.
\bibitem{hamjac} L. Ham and M. Jackson, All or Nothing: toward a promise problem dichotomy for constraint problems, in J.C. Beck (Ed.): CP 2017, LNCS 10416, pp. 139--156, 2017.
\bibitem{hamjac:hyper} L. Ham and M. Jackson, Axiomatisability and hardness for universal Horn classes of hypergraphs, Algebra Universalis (2018) 79:30.
\bibitem{how} J.M. Howie, Fundamentals of Semigroup Theory, Oxford University Press, New York; 2nd ed., 1995.
\bibitem{jac:continuum} M. Jackson, Finite semigroups whose varieties have uncountably many subvarieties, J. Algebra 228 (2000), 512--535.
\bibitem{jac:INFB} M. Jackson, Small inherently nonfinitely based finite semigroups, Semigroup Forum 64 (2002), 297--324.
\bibitem{jac:flat} M. Jackson, Flat algebras and the translation of universal Horn logic to equational logic,
J. Symb. Logic 73 (2008), 90--128.
\bibitem{jac:eqncomp} M. Jackson, Low growth equational complexity, Proc. Edinburgh Math. Soc. 62 (2019), 197--210.
\bibitem{jac:SAT} M. Jackson, Flexible constraint satisfiability and a problem in semigroup theory, arXiv:1512.03127v2.
\bibitem{jacsap} M. Jackson and O. Sapir, Finitely based, finite sets of words, Internat. J. Algebra Comput. 10 (2000), 683--708.
\bibitem{jacsto:PM} M. Jackson and T. Stokes, Identities in the algebra of partial maps, Internat. J. Algebra Comput. 16 (2006), 1131--1159.
\bibitem{jacvol} M. Jackson and M.V. Volkov, Relatively inherently nonfinitely Q-based finite semigroups, Trans. Amer. Math. Soc. 361 (2009), 2181--2206.
\bibitem{jaczha} M. Jackson and W.T. Zhang, From $A$ to $B$ to $Z$, to appear in Semigroup Forum.
\bibitem{jon} P.R. Jones, Varieties of left restriction semigroups, J. Austral. Math. Soc. 105 (2018), 173--200.
\bibitem{kunver} G. Kun, V. V\'ertesi,  The membership problem in finite flat hypergraph algebras. Internat. J. Algebra Comput. 17 (2007), no. 3, 449--459.
\bibitem{kurpol} M. Ku\v{r}il and L. Pol\'ak, On varieties of semilattice-ordered semigroups, Semigroup Forum 71 (2005),  27--48.
\bibitem{leeB0} E.W.H. Lee, Identity bases for some non-exact varieties, Semigroup Forum 68 (2004), 445--457.
\bibitem{lee} J. Leech, Inverse monoids with a natural semilattice ordering. Proc. London Math. Soc. s3-70  (1995), 146--182.
\bibitem{lyn} R. Lyndon, Identities in two-valued calculi. Trans. Amer. Math. Soc. 71 (1951), 457--465.
\bibitem{mck} R.N. McKenzie, Tarski's finite basis problem is undecidable, Internat. J. Algebra Comput. 6 (1996), 49--104.
\bibitem{mckrom} R. McKenzie and A. Romanowska, Varieties of $\cdot$-distributive bisemilattices, in Contributions to general algebra (Proc. Klagenfurt Conf., Klagenfurt, 1978), pp. 213--218, Heyn, Klagenfurt, 1979.
\bibitem{oatpow} S. Oates and M.B. Powell, Identical relations in finite groups, J. Algebra 1 (1964), 11--39.
\bibitem{ols} A. Ol$'$shanksi\u{\i}, Varieties of finitely approximable groups, Izv. Akad. Nauk SSSR Ser. Mat.
33 (1969), 915--927 [Russian; English version in Math. USSR-Izv. 3 (1969), 867--877].
\bibitem{pas} F. Pastijn,  Varieties generated by ordered bands II. Order 22 (2005), 129--143.
\bibitem{paszha} F. Pastijn and X.Z. Zhao, Varieties of idempotent semirings with commutative addition, Algebra Universalis 54 (2005),  301--321.
\bibitem{per} P. Perkins, Bases for equational theories of semigroups, J. Algebra 11 (1968) 298--314.
\bibitem{pin} J.-E. Pin, $BG=PG$, a success story. In J. Fountain (ed.), Semigroups, Formal Languages and Groups [NATO ASI Ser., Ser. C: Math. Phys. Sci. 466], Kluwer Academic Publishers, Dordrecht-Boston-London, 1995, pp. 33--37.
\bibitem{renzha} M.M. Ren and X.Z. Zhao, The varieties of semilattice-ordered semigroups satisfying $x^3 \approx x$ and
$xy \approx yx$, Period. Math. Hung. 72 (2016), 158--170.
\bibitem{RZS} M.M.Ren, X.Z. Zhao, Y. Shao, The lattice of ai-semiring varieties satisfying $x^n\approx x$ and $xy\approx yx$, Semigroup Forum 100 (2020),  542--567.
\bibitem{RZW} M.M. Ren, X.Z. Zhao, A.F. Wang, On the varieties of ai-semirings satisfying $x^3 \approx x$, Algebra Universalis 77 (2017), 395--408.
\bibitem{sap1} M.V. Sapir, Problems of Burnside type and the finite basis property in varieties of semigroups, Izv. Akad. Nauk SSSR, Ser. Mat. 51 (1987) 319--340 (in Russian).
\bibitem{sap2} M.V. Sapir, Inherently nonfinitely based finite semigroups, Mat. Sb. 133 (1987) 154--166 (in Russian).
\bibitem{osap} O. Sapir, Finitely based words, Internat. J. Algebra Comput. 10 (2000), 457--480.
\bibitem{sharen} Y. Shao, M.M. Ren, On the varieties generated by ai-semirings of order two,
\emph{Semigroup Forum}, 91 (2015) 171--184.
\bibitem{sze} Z. Sz\'ek\'ely, Computational complexity of the finite membership problem for varieties, Internat. J. Algebra Comput. 12 (2002), 811-823.
\bibitem{tra} A.N. Trahtman, An identity basis of the five-element Brandt semigroup,
Ural. Gos. Univ. Mat. Zap. 12 (1981), 147--149.
\bibitem{vol21} M.V. Volkov, Semiring identities of the Brandt monoid, to appear in Algebra Universalis, 2021.
\bibitem{wil} R. Willard, On McKenzie's method, Period. Math. Hungar. 32 (1996), 149--165.
\bibitem{ZRCSD} X.Z. Zhao, M.M. Ren, S. Crvenkovi\'c, Y. Shao, P. Dapi\'c, The variety generated by an AI-semiring of order three, Ural Math. J. 6 (2020), 117--132.
\end{thebibliography}

\end{document}